\documentclass[12pt]{article}
\usepackage{srcltx}
\usepackage{amsfonts,amssymb,mathrsfs,amsmath}
\usepackage{amsthm}

\newtheorem{theorem}{Theorem}[section]
\newtheorem{lemma}[theorem]{Lemma}
\newtheorem{proposition}[theorem]{Proposition}
\newtheorem{corollary}[theorem]{Corollary}

\theoremstyle{definition}
\newtheorem*{definition}{Definition}
\newtheorem{remark}[theorem]{Remark}
\newtheorem*{Index Convention}{Index Convention}
\newtheorem*{notation}{Notation}
\newtheorem{example}[theorem]{Example}

\def\keywords#1{\par\medskip
\noindent\textbf{Keywords.} #1}

\def\subjclass#1{{\renewcommand{\thefootnote}{}
\footnote{\emph{Mathematics Subject Classification (2010):} #1}}}

\begin{document}
\let\le=\leqslant
\let\ge=\geqslant
\let\leq=\leqslant
\let\geq=\geqslant
\newcommand{\e}{\varepsilon }
\newcommand{\f}{\varphi }
\newcommand{ \g}{\gamma}
\newcommand{\F}{{\Bbb F}}
\newcommand{\N}{{\Bbb N}}
\newcommand{\Z}{{\Bbb Z}}
\newcommand{\Q}{{\Bbb Q}}
\newcommand{\C}{{\Bbb C}}
\newcommand{\R}{\Rightarrow }
\newcommand{\W}{\Omega }
\newcommand{\w}{\omega }
\newcommand{\s}{\sigma }
\newcommand{\hs}{\hskip0.2ex }
\newcommand{\ep}{\makebox[1em]{}\nobreak\hfill $\square$\vskip2ex }
\newcommand{\Lr}{\Leftrightarrow }

\title{Frobenius groups of automorphisms\\ and their fixed points}

\markright{}

\author{{\sc E.\,I.~Khukhro}\\ \small Sobolev Institute of Mathematics, Novosibirsk, 630\,090,
Russia\\[-1ex] \small khukhro@yahoo.co.uk \\[-2ex]
\\{\sc N.\,Yu.~Makarenko}\\ \small Sobolev Institute of Mathematics, Novosibirsk, 630\,090,
Russia, and \\[-1ex] \small Universit\'{e} de Haute Alsace, Mulhouse, 68093,
France\\[-1ex] \small natalia\_makarenko@yahoo.fr\\[-2ex]
\\{\sc P.~Shumyatsky}\\ \small
Department of Mathematics, University of Brasilia, DF~70910-900, Brazil\\[-1ex]
\small pavel@mat.unb.br}

\date{}
\maketitle

\subjclass{Primary 17B40, 20D45; Secondary 17B70, 20D15, 20E36,  20F40, 22E25}

\begin{abstract}
Suppose that a finite group $G$ admits a Frobenius group of
automorphisms $FH$ with kernel $F$ and complement $H$ such that
the fixed-point subgroup of $F$ is trivial: $C_G(F)=1$. In this
situation various properties of $G$ are shown to be close to the
corresponding properties of $C_G(H)$. By using Clifford's theorem
it is proved that the order $|G|$ is bounded in terms
of $|H|$ and $|C_G(H)|$, the rank of $G$ is bounded in terms
of $|H|$ and the rank of $C_G(H)$, and
that $G$ is nilpotent if $C_G(H)$ is nilpotent. Lie ring methods
are used for bounding the exponent
and the nilpotency class
of $G$ in the case of metacyclic $FH$. The exponent of $G$ is
bounded in terms of $|FH|$ and the exponent of $C_G(H)$ by using Lazard's
Lie algebra associated with the Jennings--Zassenhaus
filtration and its connection with powerful subgroups. The
nilpotency class of $G$ is bounded in terms of $|H|$ and the
nilpotency class of $C_G(H)$ by considering Lie rings with a
finite cyclic grading satisfying a certain `selective nilpotency'
condition. The latter technique also yields similar results
bounding the nilpotency class of Lie rings and algebras with a
metacyclic Frobenius group of automorphisms, with corollaries for
connected Lie groups and torsion-free locally nilpotent groups
with such groups of automorphisms. Examples show that such nilpotency results
are no longer true for non-metacyclic Frobenius groups of automorphisms.\end{abstract}

\keywords{Frobenius group, automorphism, finite group, exponent, Lie ring, Lie algebra, Lie group, graded,
solvable, nilpotent}

\section{Introduction}

Suppose that a finite group $G$ admits a Frobenius group of
automorphisms $FH$ with kernel $F$ and complement $H$ such that
the fixed-point subgroup (which we call the centralizer) of $F$ is
trivial: $C_G(F)=1$. Experience shows that many properties of $G$
must be close to the corresponding properties of $C_G(H)$. For
example, when $GF$ is also a Frobenius group with kernel $G$ and
complement $F$ (so that $GFH$ is a double Frobenius group), the
second and third authors \cite{mak-shu10} proved that the
nilpotency class of $G$ is bounded in terms of $|H|$ and the
nilpotency class of $C_G(H)$. This result solved in the
affirmative Mazurov's problem 17.72(a) in Kourovka Notebook
\cite{kour}.

In this paper we derive properties of $G$ from the corresponding
properties of $C_G(H)$ in more general settings, no longer assuming
that $GF$ is also a Frobenius group. Note also that the order of $G$ is not assumed to be
coprime to the order of $FH$.

By using variations on Clifford's theorem it is
shown that  the order $|G|$ is bounded in terms
of $|H|$ and $|C_G(H)|$, the rank of $G$ is bounded in terms
of $|H|$ and the rank of $C_G(H)$,
and that $G$ is nilpotent if $C_G(H)$ is nilpotent.
By using various Lie ring methods bounds for the exponent and the
nilpotency class of $G$ are obtained in
the case of metacyclic $FH$. For bounding the exponent, we use Lazard's
Lie algebra associated with the Jennings--Zassenhaus
filtration and its connection with powerful subgroups. For bounding the
nilpotency class we consider Lie rings with a
finite cyclic grading satisfying a certain condition of `selective
nilpotency'. The latter technique yields similar
results giving nilpotency of bounded nilpotency class (also known as nilpotency index)
of Lie rings and
algebras with a metacyclic Frobenius group of automorphisms, with
corollaries for connected Lie groups and
torsion-free locally nilpotent groups with such groups of
automorphisms. Examples show that such nilpotency results
are no longer true for non-metacyclic Frobenius groups of automorphisms.

We now describe the results and the structure of the paper in more
detail. Recall that we consider a finite group $G$ admitting a
Frobenius group of automorphisms $FH$ with kernel $F$ and
complement $H$ such that $C_G(F)=1$. It is worth noting from the outset that
since $F$ is nilpotent being
a Frobenius kernel, the condition $C_G(F)=1$ implies the
solvability of $G$ by a theorem of Belyaev and Hartley \cite{ha-be}
based on the classification of finite simple groups.

In \S\,2 we begin establishing the connection between the
properties of $G$ and $C_G(H)$ by proving that the orders satisfy
the equation $|G|=|C_G(H)|^{|H|}$, the rank of $G$ is bounded in
terms of $|H|$ and the rank of $C_G(H)$, and that $G$ is nilpotent
if $C_G(H)$ is nilpotent (Theorem~\ref{t-orn}). These results are
proved by using Clifford's theorem on the basis of information
about the fixed points of $H$ in $FH$-invariant sections of $G$.
In particular, we prove that these fixed points are the images of
elements of $C_G(H)$ (Theorem~\ref{t-fp}), which is a non-trivial
fact since the orders of $G$ and $H$ are not assumed to be
coprime.

In \S\,3 we deal with bounding the exponent. First we develop the
requisite Lie ring technique, some of which was used earlier by
the first and third authors in~\cite{khushu}. We define Lazard's
Lie algebra $L_p(P)$ associated with the Jennings--Zassenhaus
filtration of a finite $p$-group $P$ and recall Lazard's
observation that elements of $P$ of order $p^k$ give rise to
elements of $L_p(P)$ that are ad-nilpotent of index $p^k$. A
useful property is the existence of a powerful subgroup of $P$ of
index bounded in terms of the number of generators of $P$ and the
nilpotency class of $L_p(P)$. A key lemma
gives a
bound for the nilpotency class of a finitely generated solvable Lie
algebra `saturated' with ad-nilpotent elements. The main result of the section is
Theorem~\ref{main}: if a finite group $G$ admits a
Frobenius group of automorphisms $FH$ with cyclic kernel $F$ and
complement $H$ such that $C_G(F)=1$, then the exponent of $G$ is
bounded in terms of $|FH|$ and the exponent of $C_G(H)$. At
present it is unclear, even in the case where $GFH$ is a double
Frobenius group, if the bound can be made independent of $|F|$,
which would give an affirmative answer to part (b) of Mazurov's
problem 17.72 in~\cite{kour}. In the proof of Theorem~\ref{main} a reduction to finite
$p$-groups is given by Dade's theorem~\cite{dade}. Then nilpotency of
Lazard's Lie algebra gives a reduction to powerful $p$-groups, to
which a lemma from \S\,2 about fixed points of $H$ is applied.

In \S\,4 a different Lie ring theory is developed, which is used
later in \S\,5 for bounding the nilpotency class of groups and Lie rings
with a metacyclic Frobenius group $FH$ of automorphisms.
This theory is stated in terms of a Lie ring $L$ with a
finite cyclic grading (which naturally arises from the `eigenspaces'
for $F$).
The condition of the fixed-point subring $C_L(H)$ being nilpotent of
class $c$ implies certain restrictions on the
commutation of the
grading
components, which we nickname `selective nilpotency'. For example,
in \cite{khu09} it was shown that if $C_L(F)=0$,
$c=1$, and $|F|$ is a prime, then each component commutes
with all but at most $(c,|H|)$-boundedly many components, which in
turn implies a strong bound for the nilpotency class of~$L$. For greater values
of $c$ more complicated `selective nilpotency' conditions naturally arise;
similar conditions were exploited earlier in the paper
\cite{mak-shu10} on double Frobenius groups.

In \S\,5 we obtain bounds for the
nilpotency class of groups and Lie rings with
metacyclic Frobenius groups of automorphisms. (Examples show that such results
are no longer true for non-metacyclic Frobenius groups of automorphism.)
The main result for finite groups is Theorem~\ref{group-nilpotency}: if a finite group
$G$ admits a Frobenius group of automorphisms $FH$ with cyclic
kernel $F$ and complement $H$ such that $C_G(F)=1$ and $C_G(H)$ is
nilpotent of class $c$, then $G$ is nilpotent of class bounded in terms
of $c$ and $|H|$  only. The proof is based on the analogous result
for Lie rings (Theorem~\ref{liering}).

We state separately the result for Lie algebras as Theorem~\ref{liealg}:
if a Lie algebra $L$ over any field admits a
Frobenius group of automorphisms $FH$ with cyclic kernel $F$ such
that $C_L(F)=0$ and $C_L(H)$ is nilpotent of class $c$, then $L$
is nilpotent of class bounded in terms of $c$ and $|H|$ only.
(Here $C_L(F)$ and $C_L(H)$ denote the fixed-point subalgebras for
$F$ and $H$.) One corollary of this theorem is for connected Lie
groups with  metacyclic Frobenius groups of automorphisms
satisfying similar conditions (Theorem~\ref{gr-lie}). Another
application is for torsion-free locally nilpotent groups with such
groups of automorphisms (Theorem~\ref{torsion-free}).

\medskip

The induced group of automorphisms of an
invariant section is often denoted by the same letter (which is a
slight abuse of notation as the action may become non-faithful).
We use the abbreviation, say, ``$(m,n)$-bounded'' for ``bounded
above in terms of~$m,\,n$ only''.

\section{Fixed points of Frobenius complements}

We begin with a
theorem of Belyaev
and Hartley \cite{ha-be} based on the classification
of finite simple groups (see also \cite{ku-shu}).

\begin{theorem}[{\cite[Theorem~0.11]{ha-be}}]\label{razr}
Suppose that a finite group $G$ admits a
nilpotent group of automorphisms $F$ such that $C_G(F)=1$. Then $G$ is solvable.
\end{theorem}

We now discuss the question of covering the fixed points of a
group of automorphisms in an invariant quotient by the fixed
points in the group. Let $A\leq {\rm Aut\,}G$ for a finite group
$G$ and let $N$ be a normal
$A$-invariant subgroup of $G$. It is well-known that if $(|A|,|N|)=1$, then
$C_{G/N}(A)=C_G(A)N/N$.
If we do not
assume that $(|A|,|N|)=1$, the equality $C_{G/N}(A)=C_G(A)N/N$ may no longer
be true. However there are some important cases when it does hold.
In particular, we have the following lemma.

\begin{lemma}\label{l-car}
Let $G$ be a finite group admitting a
nilpotent group of automorphisms $F$ such that $C_G(F)=1$.
If $N$ is a normal $F$-invariant subgroup of $G$, then $C_{G/N}(F)=\nobreak 1$.
\end{lemma}

\begin{proof}
Since $F$ is a Carter subgroup of $GF$, it follows that $NF/N$ is a
Carter subgroup of $GF/N$. Hence $C_{G/N}(F)=1$.
\end{proof}

The following theorem was proved in \cite{khu10} under the additional coprimeness assumption
 $(|N|,|F|)=1$, so here we only have to provide a reduction to this case.

\begin{theorem} \label{t-fp}
Suppose that a finite group $G$ admits a
Frobenius group of automorphisms $FH$ with kernel
$F$ and complement $H$. If $N$ is an
$FH$-invariant normal subgroup of $G$ such that
$C_N(F)=1$, then $C_{G/N}(H)=C_G(H)N/N$.
\end{theorem}

\begin{proof}
As a Frobenius kernel, $F$ is nilpotent. Hence $N$
is solvable by Theorem~\ref{razr}.

Consider an unrefinable $FH$-inva\-ri\-ant normal series of $G$
\begin{equation*}
G>N=N_1 >N_2>\dots >N_k> N_{k+1}=1
\end{equation*}
connecting $N$ with $1$; its factors $N_i/N_{i+1}$ are elementary
abelian. We apply induction on $k$ to find an element of $C_G(H)$
in any $gN\in C_{G/N}(H)$.

For $k>1$ consider the quotient $G/N_k$ and
the induced group of automorphisms $FH$.
By Lemma~\ref{l-car}, $C_{G/N_k}(F)=1$. By induction there is $c_1N_k\in C_{G/N_k}(H)\cap
gN/N_k$, and it remains to find a required element $c\in
C_G(H)\cap c_1N_k\subseteq C_{G}(H)\cap gN$. Thus the proof of the
induction step will follow from the case $k=1$.

Let $k=1$; then $N=N_k$ is a $p$-group for some prime $p$. Let
$F=F_p\times F_{p'}$, where $F_p$ is the Sylow $p$-subgroup
of~$F$. Since $C_{N}(F_{p'})$ is $F_p$-invariant, we must actually
have $C_{N}(F_{p'})=1$. Indeed, otherwise the $p$-group $F_p$
would have non-trivial fixed points on the $p$-group
$C_{N}(F_{p'})$, and clearly $C_{C_{N}(F_{p'})}(F_p)=C_{N}(F)$
(this argument works even if $F_p=1$). Thus, the hypotheses of the
theorem also hold for $G$ with the Frobenius group of automorphisms
$F_{p'}H$ satisfying the additional condition $(|N|,|F_{p'}|)=1$. Now
\cite[Theorem~1]{khu10} can be applied to produce a required
fixed point.
\end{proof}

We now prove a few useful lemmas about a finite group with a
Frobenius group of automorphisms.

\begin{lemma}\label{l-gen}
Suppose that a finite group $G$ admits a
Frobenius group of automorphisms $FH$ with kernel
$F$ and complement $H$ such that $C_G(F)=1$.
Then $G=\langle C_G(H)^f\mid f\in F\rangle$.
\end{lemma}

\begin{proof}
The group $G$ is
solvable by Theorem~\ref{razr}. Consider an
unrefinable $FH$-invariant normal series
\begin{equation}\label{riad}
G=G_1 >G_2>\dots >G_k> G_{k+1}=1.
\end{equation}
It is clearly sufficient to prove that every factor $S=
G_i/G_{i+1}$ of this series is covered by $\langle
C_{G_{i}}(H)^f\mid f\in F\rangle$, that is, $$\langle C_{G_i}(H)^f\mid
f\in F\rangle G_{i+1}/G_{i+1}=G_{i}/G_{i+1}.$$
By
Theorem~\ref{t-fp}, this is the same as $\langle C_S(H)^f\mid f\in
F\rangle=S$. Recall that $C_S(F)=1$ by Lemma~\ref{l-car}. Then
Clifford's theorem can be applied to show that $C_S(H)\ne 1$.

Recall that for a group $A$ and a field $k$, a \textit{free
$kA$-module of dimension $n$} is a direct sum of $n$ copies of the
group algebra $kA$, each of which can be regarded as a vector
space over $k$ of dimension $|A|$ with a basis $\{v_g\mid g\in
A\}$ labelled by elements of $A$ on which $A$ acts in a regular
representation: $v_gh=v_{gh}$. By the Deuring--Noether theorem
\cite[Theorem~29.7]{cur-rai} two representations over a smaller
field are equivalent if they are equivalent over a larger field;
therefore being a free $kA$-module is equivalent to being a free
$\bar{k}A$-module for any field extension $\bar k\supseteq k$, as
the corresponding permutational matrices are defined over the prime field.

We denote by ${\Bbb F}_p$ the field of $p$ elements.

\begin{lemma}\label{l-free}
Each factor $S$ of \eqref{riad}
is a free ${\Bbb F}_pH$-module for the appropriate prime~$p$.
\end{lemma}

\begin{proof}
Again, we only provide reduction to the coprime case considered in \cite{khu10}.
Let $S$ be an elementary $p$-group; then let $F=F_p\times
F_{p'}$ as in the proof of Theorem~\ref{t-fp}. As therein, we must
actually have $C_{S}(F_{p'})=1$. Refining $S$ by a non-refinable
$F_{p'}H$-invariant normal series we obtain factors that are
irreducible ${\mathbb F}_pF_{p'}H$-modules. Having the additional
condition that $p\nmid |F_{p'}|$ we can now apply, for example,
\cite[Lemma~2]{khu10} to obtain that each of them is a free
${\mathbb F}_pH$-module, and therefore $S$ is also a free
${\mathbb F}_pH$-module.
\end{proof}

We now finish the proof of Lemma~\ref{l-gen}. By Lemma~\ref{l-free}, $S$ is a free
${\Bbb F}_pH$-module, which means that $S=\bigoplus _{h\in H}Th$ for some ${\Bbb F}_pH$-sub\-mo\-dule $T$.
Hence, $C_S(H)\ne 0$, as $0\ne \sum _{h\in H}th\in C_S(H)$ for any $0\ne t\in T$.
Since the series \eqref{riad} is non-refinable, the ${\mathbb
F}_pFH$-module $S$ is irreducible. Therefore,
$$0\ne \langle C_S(H)^{HF}\rangle
=\langle C_S(H)^{F}\rangle =S,$$
which is exactly what we need.
\end{proof}

\begin{lemma}\label{sylow}
Suppose that a finite group $G$ admits a Frobenius group of
automorphisms $FH$ with kernel $F$ and complement $H$ such that
$C_G(F)=1$. Then for each prime $p$ dividing $|G|$ there is a
unique $FH$-invariant Sylow $p$-subgroup of $G$.
\end{lemma}

\begin{proof}
Recall that $G$ is solvable, and so is $GF$. Since $F$ is a Carter
subgroup of $GF$, it contains a system normalizer of $G$. By
P.~Hall's theorem \cite[Theorem~9.2.6]{robinson}, a system normalizer covers all central factors
of any chief series of $GF$. Since $F$ is nilpotent, it follows
that $F$ is a system normalizer. Furthermore, $F$ normalizes a
unique Sylow $p$-subgroup. Indeed, if $P$ and $P^g$ for $g\in G$
are two Sylow $p$-subgroups normalized by $F$, then $P$ is
normalized by $F$ and $F^{g^{-1}}$. Then $F$ and $F^{g^{-1}}$ are
Carter subgroups of $N_G(P)$ and $F=F^{g^{-1}n}$ for some
$n\in N_G(P)$, whence $g^{-1}n=1$, as $N_G(F)=C_G(F)=1$. Thus, $P^g=P^n=P$.
Since $F$ is normal in $FH$, the uniqueness of $P$ implies that it is also
$H$-invariant.
\end{proof}

We now establish the connection between the order, rank, and nilpotency of $G$ and $C_G(H)$ for
a finite group $G$ admitting a Frobenius group of automorphisms $FH$ with fixed-point-free kernel
$F$. (By the rank we mean the minimum number $r$ such that every subgroup can be generated by $r$ elements.)

\begin{theorem} \label{t-orn}
Suppose that a finite group $G$ admits a
Frobenius group of automorphisms $FH$ with kernel
$F$ and complement $H$ such that $C_G(F)=1$. Then

{\rm (a)} $|G|=|C_G(H)|^{|H|}$;

{\rm (b)} the rank of $G$ is
bounded in terms of $|H|$ and the rank of $C_G(H)$;

{\rm (c)} if $C_G(H)$ is nilpotent, then $G$ is nilpotent.
\end{theorem}

\begin{proof}
(a) It is sufficient to prove this equality for each factor
$S=G_i/G_{i+1}$ of the series \eqref{riad}, since $|C_G(H)|=\prod _i|C_{G_i/G_{i+1}}(H)|$
by Theorem~\ref{t-fp}. By Lemma~\ref{l-free}, $S$ is a free
${\Bbb F}_pH$-module, which means that $S=\bigoplus _{h\in H}Th$ for some ${\Bbb F}_pH$-sub\-mo\-dule~$T$.
Therefore, $C_S(H)=\{\sum _{h\in H}th\mid t\in T\}$ and $|C_S(H)|=|T|$, whence $|S|=|T|^{|H|}$.

(b) It is known that the rank of a finite (solvable) group $G$ is
bounded in terms of the maximum rank of its Sylow subgroups
\cite{kov68}. Let $P$ be an $FH$-invariant Sylow $p$-subgroup of $G$ given by Lemma~\ref{sylow}.
It is known that the rank of a $p$-group of automorphisms of a
finite $p$-group $U$ is bounded in terms of the rank of $U$. Let
$U$ be a Thompson critical subgroup of $P$; recall that $U$ is a
characteristic subgroup of nilpotency class at most 2 containing its
centralizer in $P$ (see, for example, \cite[Theorem~5.3.11]{gor}).
Thus the rank of $P$ is bounded in terms of the
rank of $U$. In turn, since $U$ is nilpotent of class at most 2, the rank
of $U$ is bounded in terms of the rank of $S=U/\Phi (U)$. The
group $S$ can be regarded as an ${\mathbb F}_pFH$-module, which is a
free ${\mathbb F}_pH$-module by a repeated application of
Lemma~\ref{l-free} to an unrefinable series of ${\mathbb
F}_pFH$-sub\-mo\-dules of $U/\Phi (U)$. By the same argument as in the
proof of (a) above, the rank of $S$ is equal to $|H|$ times the rank of
$C_S(H)$. By Theorem~\ref{t-fp}, $C_S(H)$ is covered by
$C_G(H)$; as a result, the rank of $S$ is at most $|H|$ times the
rank of $C_G(H)$.

(c) First we make a simple remark that in any action of the
Frobenius group $FH$ with non-trivial action of $F$ the complement
$H$ acts faithfully. Indeed, the kernel $K$ that does not contain
$F$ must intersect $H$ trivially: $K\cap H$ acts
trivially on $F/(K\cap F)\ne 1$ and therefore has non-trivial fixed
points on $F$, as the action is coprime.

Suppose that $C_G(H)$ is nilpotent.
We prove that then $G$ is nilpotent by contradiction, considering a counterexample $GFH$
of minimal possible order. Recall that $G$
is solvable. Suppose that $G$ is not nilpotent; then it is easy to find an $FH$-invariant
section $VU$ of $G$ such that $V$ and
$U$ are elementary abelian groups of coprime orders, $V$ is normal
in $VUFH$, and $U$ acts faithfully on $V$ with $C_V(U)=1$. Note
that $C_{VU}(F)=1$ by Lemma~\ref{l-car}. In particular, $F$ acts non-trivially on $VU$
and therefore $H$ acts faithfully on $VU$. Since $C_{VU}(H)$ is covered by $C_G(H)$ by
Theorem~\ref{t-fp}, it follows that $C_{VU}(H)$ is nilpotent.
Thus, we can replace $G$ by $UV$, and $F$ by its image in its action on $UV$, so
by the minimality of our counterexample we must actually have $G=VU$.

Also by
Theorem~\ref{t-fp}, $C_{VU}(H)=C_V(H)C_U(H)$.
Furthermore, since $U$ and $V$ have coprime orders, the nilpotency of
$C_{VU}(H)$ implies that it is abelian; in other words, $C_U(H)$ centralizes $C_{V}(H)$.
Note that $C_U(H)\ne 1$ by Lemma~\ref{l-gen} (or \ref{l-free}).

Let $V$ be an elementary $p$-group; we can regard $V$ as an
${\Bbb F}_pUFH$-module.
Note that ${V}$ is a free ${\Bbb F}_pH$-module by Lemma~\ref{l-free}.
We extend the ground field ${\Bbb F}_p$ to a finite field
$\overline{{\Bbb F}}_p$ that is a splitting field for $UFH$ and obtain an $\overline{{\Bbb
F}}_pUFH$-module $\widetilde V=V\otimes _{{\Bbb F}_p}\overline{{\Bbb
F}}_p$. Many of the above-mentioned properties of $V$ are inherited by ${\widetilde{V}}$:
\begin{enumerate}\itemsep-0.5ex

\item[(V1)] \ $\widetilde{V}$ is a faithful $\overline{{\Bbb F}}_pU$-module;

\item[(V2)] \ $C_{\widetilde{V}}(U)=0$;

\item[(V3)] \ $C_U(H)$ acts trivially on $C_{\widetilde{V}}(H)$;

\item[(V4)] \ $C_{\widetilde{V}}(F)=0$;

\item[(V5)] \ $\widetilde{V}$ is a free $\overline{{\Bbb F}}_pH$-module.

\end{enumerate}

Consider an unrefinable series of $\overline{{\Bbb
F}}_pUFH$-sub\-mo\-dules
\begin{equation}\label{riad2}
\widetilde{V}=V_1 >V_2>\dots >V_k> V_{k+1}=0.
\end{equation}
Let $W$ be one of the factors of this series; it is a non-trivial
irreducible $\overline{{\Bbb F}}_pUFH$-sub\-mo\-dule. Note that $C_W(F)=0$
by Lemma~\ref{l-car}, as we can still regard $\widetilde V$ as a finite (additive)
group on which $F$ acts fixed-point-freely by property~(V4).
For the same reason,
$W$ is a free ${ \overline {\Bbb
F}}_pH$-module by Lemma~\ref{l-free}.

Let $$W=W_1\oplus \dots \oplus W_t$$ be the decomposition of $W$
into the direct sum of Wedderburn components $W_i$ with respect to
$U$. On each of the $W_i$ the group $U$ is represented by scalar multiplications.
We consider the transitive action of $FH$ on the set
$\Omega=\{W_1,\dots ,W_t\}$.

\begin{lemma}\label{reg-orb}
All orbits of
$H$ on $\Omega$, except for possibly one, are regular (that is, of length~$|H|$).
\end{lemma}

\begin{proof}
First note that $H$ transitively permutes the $F$-orbits on $\Omega$. Let
$\Omega _1$ be one of these $F$-orbits and let $H_1$ be the
stabilizer of $\Omega _1$ in $H$. If $H_1=1$, then all the
$H$-orbits are regular, so we assume that $H_1\ne 1$.
We claim that $H_1$ has exactly one
non-regular orbit on~$\Omega _1$ (actually, a fixed point).

Let $\bar F$ be the image of $F$
in its action on $\Omega _1$. If $\bar F=1$, then $\Omega
_1$ consists of a single Wedderburn component, on which
$U$ acts by scalar multiplications, and acts non-trivially by
property~(V2). Then $F$ acts trivially on the non-trivial quotient
of $U$ by the corresponding kernel, which contradicts Lemma~\ref{l-car}. Thus, $\bar F\ne 1$, and
by the remark at the beginning of the proof, $\bar{F}H_1$ is  a
Frobenius group with complement~$H_1$.

Let $S$ be the stabilizer of a point in $\Omega _1$ in
$\bar{F}H_1$. Since $|\Omega _1|=|\bar F:\bar F\cap S|=|\bar F H_1:S|$ and the
orders  $|\bar F|$ and $|H_1|$ are coprime, $S$ contains a conjugate of
$H_1$; without loss of generality we assume that $H_1\leq S$. Any
other stabilizer of a point is equal to $S^f$ for $f\in \bar F\setminus
S$. We claim that $S^f\cap H_1=1$, which is the same as $S\cap
H_1^{f^{-1}}=1$. But all the conjugates $H_1^x$ for $x\in \bar F$ are
distinct and disjoint and their union contains all the elements of $\bar FH_1$ of
orders dividing $|H_1|$; the same is true for the conjugates of $H_1$ is $S$.
Therefore the only conjugates of $H_1$
intersecting $S$ are $H_1^s$ for $s\in S$.

The $H$-orbits of elements of regular $H_1$-orbits are regular
$H$-orbits. Thus there is one non-regular $H$-orbit on
$\Omega$  --- the $H$-orbit of the fixed point of $H_1$ on $\Omega _1$.
\end{proof}

Consider any regular $H$-orbit on $\Omega$, which we temporarily
denote by $\{W_h \mid h\in H\}$. Let $X=\bigoplus _{h\in H} W_h$. Then,
as before, $C_X(H)=\{\sum _{h\in H}xh\mid x\in W_1\}$. Since $C_U(H)$, as a
subgroup of $U$, acts on each $W_h$ by scalar multiplications and
centralizes $C_X(H)$ by property~(V3), we obtain that, in fact, $C_U(H)$ must act trivially on $X$.

The sum of the $W_i$ over all regular $H$-orbits is obviously a
free $\overline{{\Bbb F}}_pH$-module. Since $\widetilde V$ is also
a free $\overline{{\Bbb F}}_pH$-module by property~(V5), the sum $Y$ of the $W_i$
over the only one, by Lemma~\ref{reg-orb}, possibly remaining non-regular $H$-orbit must also be a free
$\overline{{\Bbb F}}_pH$-module. Let $Y=\bigoplus _{h\in H} Zh$
for some $\overline{{\Bbb F}}_pH$-sub\-mo\-dule $Z$. Then, as before,
$C_Y(H)=\{\sum _{h\in H}yh\mid y\in Z\}$. The subgroup $C_U(H)$
acts on each $W_i$ by scalar multiplications. Since $Y$ is the sum
over an $H$-orbit and $H$ centralizes $C_U(H)$, all the $W_i$ in the $H$-orbit
are isomorphic $\overline{{\Bbb F}}_pC_U(H)$-modules. Hence $C_U(H)$
acts by scalar multiplications on the whole $Y$. Since $C_U(H)$
centralizes $C_Y(H)$ by property~(V3), it follows that, in fact, $C_U(H)$
must act trivially on $Y$.

As a result, $C_U(H)$ acts trivially on $W$. Since this is true
for every factor of \eqref{riad2} and the order of $U$ is coprime
to the characteristic $p$ (or to the order of $\widetilde V$), it
follows that $C_U(H)$ acts trivially on $\widetilde V$, contrary to
property~(V1).
This contradiction completes the proof.
\end{proof}

\section{Bounding the exponent}

Here we bound the exponent of a group with a metacyclic Frobenius
group of automorphisms. But first we develop the requisite Lie
ring technique. The following general definitions
are also used in subsequent sections.

The Lie subring (or subalgebra)
generated by a subset~$U$ is denoted by $\left<U\right>$, and the ideal generated by~$U$ by ${}_{\rm
id}\!\left<U\right>$. Products in a Lie ring are called commutators. We use the term ``span''
both for the subspace (in the case of algebras) and for the additive subgroup generated by a
given subset. For subsets $X,Y$ we denote
by $[X,Y]$ the span of all commutators $[x,y]$, $x\in X$, $y\in Y$;
this is an ideal if $X,Y$ are ideals. Terms of the derived series of a Lie ring $L$ are defined as $\;
L^{(0)}=L;$ \ $L^{(k+1)}=[L^{(k)},\, L^{(k)}].$ Then $L$ is
solvable of derived length at most $n$ (sometimes called
``solvability index'') if $L^{(n)}=0$. Terms of the lower central
series of $L$ are defined as $\gamma _1(L)=L;$ \ $\gamma
_{k+1}(L)=[\gamma _k(L),\,L]$. Then $L$ is nilpotent of class at most $c$
(sometimes called ``nilpotency index'') if $\gamma _{c+1}(L)=0$.
We use the standard notation for simple (left-normed) commutators:
$[x_1,x_2,\dots , x_k]=[...[[x_1,x_2],x_3],\dots , x_k]$ (here the $x_i$ may be elements or subsets).

We use several times  the following fact, which helps in proving
nilpotency of a solvable Lie ring $L$: : let $K$ be an ideal of a
Lie ring $L$;
\begin{equation}\label{chao}
\text{if}\;\, \g _{c+1}(L)\subseteq [K,K]\;\,\text{and}\;\,
\g _{k+1}(K)=0,\;\, \text{then}\;\,\g _{c{k+1 \choose 2} -{k \choose 2}+1}(L)=0.\end{equation}
This can be regarded as a Lie ring analogue of P.~Hall's theorem \cite{hall} (with a simpler `linear' proof;
see also \cite{stewart} for the best possible bound for the nilpotency class of $L$).

\medskip
Let $A$ be an additively written abelian group. A Lie ring $L$ is
\textit{$A$-graded} if
$$L=\bigoplus_{a\in A}L_a\qquad \text{ and }\qquad[L_a,L_b]\subseteq L_{a+b},\quad a,b\in A,$$
where the $L_a$ are subgroups of the additive group of $L$.
Elements of the grading components $L_a$ are called \textit{homogeneous}, and commutators in homogeneous elements
\textit{homogeneous commutators}. An additive subgroup $H$ of $L$ is
called \textit{homogeneous} if $H=\bigoplus_a (H\cap L_a)$; we then
write $H_a=H\cap L_a$. Clearly, any subring or ideal generated by
homogeneous additive subgroups is homogeneous. A homogeneous
subring and the quotient by a homogeneous ideal can be regarded as
$A$-graded Lie rings with the induced gradings. Also, it is not
difficult to see that if $H$ is homogeneous, then so is $C_L(H)$, the centralizer of $H$,
which is, as usual, equal to the set $\{l\in L\mid
[l,h]=0 \text{ for all } h\in H\}$.

An element $y$ of a Lie algebra $L$ is called
\textit{ad-nilpotent} if there exists a positive integer $n$ such
that $({\rm ad}\, y)^n=0$, that is,
$[x,\underbrace{y,\dots,y}_{n}]=0$ for all $x\in L$. If $n$ is the
least integer with this property, then we say that $y$ is
\textit{ad-nilpotent of index}~$n$.

\medskip
Throughout the rest of this section $p$ will denote an arbitrary but fixed prime.
Let $G$ be a group.
We set
$$D_i=D_i(G)=\prod\limits_{jp^k\geq i}\gamma_j(G)^{p^k}.$$
The
subgroups $D_i$
form the \textit{Jennings--Zassenhaus filtration}
$$G=D_1\geq D_2\geq\cdots $$
of the group $G$.
This series satisfies the inclusions
$[D_i,D_j]\leq D_{i+j}$ and $D_i^p\leq D_{pi}$ for all $i,j$.
These properties make it possible to construct a Lie algebra $DL(G)$ over $\mathbb F_p$,
the field with $p$ elements. Namely, consider the quotients $D_i/D_{i+1}$ 
as linear spaces over
$\mathbb F_p$, and let $DL(G)$ be the direct
sum of these spaces. Commutation in $G$ induces a binary
operation $[\, ,]$ in $DL(G)$. For homogeneous elements
$xD_{i+1}\in D_i/D_{i+1}$ and $yD_{j+1}\in D_j/D_{j+1}$
the operation is defined by
$$[xD_{i+1},yD_{j+1}]=[x,y]D_{i+j+1}\in D_{i+j}/D_{i+j+1}$$
and extended to arbitrary
elements of $DL(G)$ by linearity. It is easy
to check that the operation is
well-defined and that $DL(G)$ with the operations $+$ and $[\, ,]$
is a Lie algebra over $\mathbb F_p$, which is $\mathbb Z$-graded with the $D_i/D_{i+1}$
being the grading components.

For any $x\in D_i\setminus D_{i+1}$ let $\bar x$ denote the element
$xD_{i+1}$ of $DL(G)$.

\begin{lemma}[Lazard \cite{la}]\label{3.6}
For any $x\in G$ we have
$({\rm ad}\, \bar x)^p={\rm ad}\, \overline{x^p}$. Consequently, if $x$ is of finite order
$p^t$, then $\bar x$ is ad-nilpotent of index at most $p^t$.
\end{lemma}

 Let
$L_p(G)=\langle D_1/D_2\rangle$ be the subalgebra
of $DL(G)$ generated by $D_1/D_2$; 
it is also $\mathbb Z$-graded with grading components $ L_i=L_p(G)\cap D_i/D_{i+1}$. 
 The following lemma
goes back to Lazard \cite{laz65}; in the present form it can be found, for example,
 in \cite{khushu}.

\begin{lemma}\label{4.9}
Suppose that $X$ is a
$d$-generator finite $p$-group such that the Lie
algebra $L_p(X)$ is nilpotent of class $c$. Then
$X$ has a powerful characteristic subgroup of
$(p,c,d)$-bounded index.
\end{lemma}
Recall that powerful $p$-groups were introduced by
Lubotzky and Mann in \cite{lbmn}: a finite $p$-group
$G$ is \textit{powerful} if $G^p\geq [G,G]$ for
$p\ne 2$ (or $G^4\geq [G,G]$ for $p=2$). These groups
have many nice properties, so that often a problem becomes
much easier once it is reduced to the case of powerful
$p$-groups. The above lemma is quite useful as it allows
us to perform such a reduction. We will also require the following
lemma. Note that when a $\Bbb Z$-graded
Lie algebra $L$ is generated by the component $L_1$,
that is, $L=\langle L_1\rangle$, then in fact $L=L_1\oplus L_2\oplus \cdots$.

\begin{lemma}\label{lie}
Let $L=\bigoplus L_i$ be a $\Bbb Z$-graded
Lie algebra over a field such that $L=\langle L_1\rangle$ and assume that
every homogeneous component $L_i$ is spanned by elements that are
ad-nilpotent of index at most $r$. Suppose further that $L$ is solvable of
derived length $k$ and that $L_1$ has finite dimension $d$. Then $L$ is
nilpotent of $(d,r,k)$-bounded class.
\end{lemma}
\begin{proof}
Without loss of generality we can assume that $k\geq 2$
and use induction on $k$. Let $M$ be the last non-trivial term of the
derived series of $L$. By induction we assume that $L/M$ is nilpotent of
$(d,r,k)$-bounded class. In particular, it follows that the dimension
of $L/M$ is finite and $(d,r,k)$-bounded. If $M\leq Z(L)$, we deduce
easily that $L$ is nilpotent of $(d,r,k)$-bounded class. So assume that
$M\not\leq Z(L)$. Let $j$ be the biggest index such that $L_j$ is not
contained in $C_L(M)$. Thus, $K=L_j+C_L(M)$ is a non-abelian ideal
in $L$. Since the dimension of $L/M$ is finite and $(d,r,k)$-bounded,
there exist boundedly many elements $a_1,\dots,a_m\in L_j$
such that $K=\langle a_1,\dots,a_m,C_L(M)\rangle$ and each of the
elements $a_1,\dots,a_m$ is ad-nilpotent of index at most $r$.
Taking into account that $j$ here is $(d,r,k)$-bounded, we can use
backward induction on $j$. Therefore, by induction $L/[K,K]$ is nilpotent
of $(d,r,k)$-bounded class.

Set $s=m(r-1)+1$ and consider the ideal $S=[M,K,\dots,K]$, where $K$ occurs
$s$ times. Then $S$ is spanned by commutators of the form
$[m,b_1,\dots,b_s]$ where $m\in M$ and $b_1,\dots,b_s$ are not necessarily
distinct elements from $\{a_1,\dots,a_m\}$. Since the number $s$ is big
enough, it is easy to see that there are $r$ indices $i_1,\dots,i_r$ such
that $b_{i_1}=\dots=b_{i_r}$. Taking into account that $K/C_L(M)$ is
abelian, we remark that $[m,x,y]=[m,y,x]$ for all $m\in M$ and $x,y\in K$.
Therefore we can assume without loss of generality that $b_1=\dots=b_r$. Since
$b_1$ is ad-nilpotent of index at most $r$, it follows that
$[m,b_1,\dots,b_s]=0$. Thus, $S=0$, which means that $M$ is contained in the
$s$th term of the upper central series of $K$. The fact that $L/M$ is
nilpotent of $(d,r,k)$-bounded class now implies that $K$ is nilpotent
of $(d,r,k)$-bounded class as well. Combining this with the fact that
 $L/[K,K]$ is also nilpotent of $(d,r,k)$-bounded class and using the Lie
ring analogue \eqref{chao} of P.~Hall's theorem  we deduce
that $L$ is also nilpotent of $(d,r,k)$-bounded class. The proof is complete.
\end{proof}

We now prove the main result of this section.

\begin{theorem} \label{main}
Suppose that a finite Frobenius group
$FH$ with cyclic kernel $F$ and complement $H$ acts on a finite group $G$ in
such a manner that $C_G(F)=1$ and $C_G(H)$ has exponent $e$. Then
the exponent of $G$ is bounded solely in terms of $e$ and $|FH|$.
\end{theorem}
\begin{proof}

By Theorem~\ref{razr} the group $G$ is solvable. The nilpotent length
(Fitting height) of $G$ is bounded in terms of $|F|$ by Dade's
theorem~\cite{dade}. Therefore it is sufficient to bound the
exponents of the factors of the Fitting series of $G$. By
Lemma~\ref{l-car} and Theorem~\ref{t-fp} each of them inherits the
hypotheses $C_G(F)=1$ and $C_G(H)^e=1$.
Therefore we can assume from the outset that $G$ is a
finite $p$-group for some prime $p$. In view of Lemma~\ref{l-gen},
$G=\langle C_G(H)^f\mid f\in F\rangle$, so $p$ divides $e$.
We assume without loss of generality that $e$ is a $p$-power.
Let $x\in G$. It is clear that $x$ is contained in an
$FH$-invariant subgroup of $G$ with at most $|FH|$ generators.
Therefore without loss of generality we can assume that $G$ is
$|FH|$-generated.

 Any group of automorphisms of the group $G$ acts
naturally on every factor of the Jennings--Zassenhaus filtration
of $G$. This action induces an action by automorphisms on the Lie algebra
$L=L_p(G)$. Lemma~\ref{l-car} shows that $F$ is fixed-point-free on every
factor of the Jennings--Zassenhaus filtration. Hence
$C_L(F)=0$. Kreknin's theorem \cite{kr} now tells us that $L$ is
solvable of $|F|$-bounded derived length.

Every homogeneous
component $L_i$ of the Lie algebra $L$ can be regarded as an $FH$-invariant
subgroup of the corresponding quotient $D_i/D_{i+1}$
of the Jennings--Zassenhaus series.
By Lemma~\ref{l-gen} we obtain that $L_i$
is generated by the centralizers $C_{L_i}(H)^f$ where
$f$ ranges through $F$. Moreover, Theorem \ref{t-fp} implies that every
element of $C_{L_i}(H)^f$ is the image of some element of $C_G(H)^f$,
the order of which divides~$e$.
It follows by Lemma~\ref{3.6} that the additive group $L_i$ is generated by elements
that are ad-nilpotent of index at most~$e$. We now deduce from Lemma~\ref{lie} that $L$
is nilpotent of $(e,|FH|)$-bounded class.

Lemma~\ref{4.9} now tells us that $G$ has a powerful
characteristic subgroup of $(e,|FH|)$-bounded index. It suffices
to bound the exponent of this powerful subgroup so we can just
assume that $G$ is powerful. Powerful $p$-groups have the property
that if a powerful $p$-group $G$ is generated by elements of order
dividing $p^k$, then the exponent of $G$ also divides $p^k$ (see
\cite[Lemma~2.2.5]{ddms}). Combining this with the fact that our
group $G$ is generated by elements of $C_G(H)^f$, $f\in F$, we conclude that
$G$ has exponent~$e$.
\end{proof}

\section{Cyclically graded Lie rings with `selective nilpotency' condition}

In this section we develop a Lie ring theory which is used in
\S\,5 for studying groups $G$ and Lie rings $L$ with a metacyclic
Frobenius group of automorphisms $FH$. This theory is stated in
terms Lie rings with finite cyclic grading, which will arise from
the `eigenspaces' for $F$. By Kreknin's theorem \cite{kr} the
condition $C_L(F)=0$ implies the solvability of $L$ of derived
length bounded in terms of $|F|$, but our aim is to obtain
nilpotency of class bounded in terms of $|H|$ and the nilpotency
class of $C_L(H)$. (Here $C_L(F)$ and $C_L(H)$ denote the
fixed-point subrings for $F$ and $H$.) The nilpotency of $C_L(H)$
of class $c$ implies certain restrictions on the commutation of
the grading components, which we nickname `selective nilpotency'.
For example, in \cite{khu09} it was shown that if $c=1$, that is,
$C_L(H)$ is abelian, and $|F|$ is a prime, then each component
commutes with all but at most $(c,|H|)$-boundedly many components,
which in turn implies a strong bound for the nilpotency class of
$L$. In this section we work with another, rather technical
`selective nilpotency' condition. As we shall see in \S\,5, this
condition actually arises quite naturally in dealing with a Lie
ring admitting a metacyclic Frobenius group of automorphisms;
similar conditions were exploited earlier in the paper
\cite{mak-shu10} on double Frobenius groups. In this section we
virtually achieve a good $(c,|H|)$-bounded derived length and Engel-type conditions, which
prepare ground for the next section, where   nilpotency of
$(c,|H|)$-bounded class is finally obtained by using additional
considerations.

\subsubsection*{Numerical preliminaries} The next two lemmas are elementary facts on polynomials.

\begin{lemma}\label{char0}
Let $K$ be a field of characteristic 0 containing
a primitive $q$th root of unity $\omega$. Suppose that $m=\omega^{i_1}+\dots+\omega^{i_m}$
for some positive integer $m$ and some $0\leq i_1,\dots,i_m\leq q-1$. Then $i_1=\dots=i_m=0$.
\end{lemma}
\begin{proof}
Without loss of generality it can be assumed that
$K=\Bbb Q[\omega]$. Since $m=|\omega^{i_1}|+\dots+|\omega^{i_m}|$,
the lemma follows from the triangle inequality.
\end{proof}

\begin{lemma}\label{charp}
Let $g_1(x)=a_sx^s+\cdots+ a_0$ and
$g_2(x)=b_tx^t+\cdots+ b_0$ be polynomials with integer
coefficients that have no non-zero common complex roots and let
$M=\max_{i,j}\{|a_i|,|b_j|\}$. Suppose that $n_0$ is a positive integer such
that $g_1$ and $g_2$ have a non-zero common root in $\Bbb
Z/n_0\Bbb Z$. Then $n_0\leq 2^{2^{(s+t-1)}-1}M^{2^{(s+t-1)}}$.
\end{lemma}

\begin{proof}
We use induction on $s+t$.
If one of the polynomials has degree $0$, say, $g_1(x)=b$,
then $n_0$ divides $|b|$ and the statement is correct.
Therefore we can assume that $s$ and $t$ are both positive and $a_s,  b_t\ne 0$.
Let $s\geq t$. We set $p(x)=b_tg_1(x)-g_2(x)a_sx^{s-t}$. The polynomial
$p(x)$ has degree at most $s-1$ and the polynomials $g_2(x)$ and
$p(x)$ have a non-zero common root in $\Bbb Z/n_0\Bbb Z$ but not
in $\Bbb C$.
 We observe that $p(x)$ is of the form
$$p(x)=c_{s-1}x^{s-1}+\cdots+c_0,$$
where
$$c_i=b_ta_i-a_sb_{i-s+t} \,\,\,\,\,\mathrm{for}\,\,\,
\,\,i=s-t, \dots, s-1,$$
$$c_i= b_ta_i \,\,\,\,\,\mathrm{for} \,\,\,\,\,i=0,\dots, s-t-1.$$
Let $M_0=\max_{i,j}\{|c_i|,|b_j|\}$. By induction, $n_0\leq
2^{2^{(s+t-2)}-1}{M_0}^{2^{(s+t-2)}}$. Using that $M_0\leq 2M^2$
we compute $$n_0\leq 2^{2^{(s+t-2)}-1}{M_0}^{2^{(s+t-2)}}\leq
2^{2^{(s+t-1)}-1}M^{2^{(s+t-1)}},$$ as required.
\end{proof}

In \S\,5 we shall consider a Frobenius group $FH$ with
cyclic kernel $F=\langle f\rangle$ of order $n$ and (necessarily cyclic) complement $H=\langle h\rangle$ of order $q$ acting
as $f^h=f^r$ for $1\leq r\leq n-1$. This is where the following conditions come from for numbers $n, q, r$, which
we fix for the rest of this section:
\begin{equation}\label{prim}
\begin{split}
 & n, q, r \text{ are positive integers such that } 1\leq r \leq n-1 \text{ and } \\
&\quad\text{the image of } r \text{ in } {\Bbb Z}/d{\Bbb Z} \text{ is a
primitive } q \text{th root of } 1 \\ &\qquad \qquad\qquad\text{for every divisor } d
\text{ of }n.
\end{split}
\end{equation}

In particular, $q$ divides $d-1$ for every
divisor $d$ of $n$. When convenient, we freely identify $r$ with its image in $\Bbb
Z/n\Bbb Z$ and regard $r$ as an element of $\Bbb Z/n\Bbb Z$ such
that $r^q=1$.

\begin{definition}
Let $a_1,\dots,a_k$ be not
necessarily distinct non-zero elements of $\Bbb Z/n\Bbb Z$. We
say that the sequence $(a_1,\dots,a_k)$ is \textit{$r$-dependent} if
$$a_{1}+\dots+a_{k}=r^{\alpha_1}a_{1}+\dots+r^{\alpha_k}a_{k}$$
for some $\alpha_i \in\{0,1,2,\dots,q-1\}$ not all of which are zero.
If the sequence $(a_1,\dots,a_k)$ is not $r$-dependent, we
call it \textit{$r$-independent}.
\end{definition}

\begin{remark}\label{remark}
A single non-zero element $a\in \Bbb Z/n\Bbb Z$ is always $r$-independent:
if $a=r^{\alpha}a$ for $\alpha \in\{1,2,\dots,q-1\}$, then $a=0$ by \eqref{prim}.
\end{remark}

\begin{notation}
For a given $r$-independent sequence $(a_1,\dots,a_k)
$ we
denote by $D(a_1,\dots,a_k)$ the set of all $j\in\Bbb Z/n\Bbb Z$
such that $(a_1,\dots,a_k,j)$ is $r$-dependent.
\end{notation}

\begin{lemma}\label{115}
If $(a_1,\dots,a_k)$ is $r$-independent, then
$|D(a_1,\dots,a_k)|\leq q^{k+1}$.
\end{lemma}
\begin{proof}
Suppose that $(a_1,\dots,a_k,j)$ is $r$-dependent. We have
$$a_1+a_2+\dots+a_k+j=r^{i_1}a_1+r^{i_2}a_2+\dots+r^{i_k}a_k+r^{i_0}j$$
for suitable $0\leq i_s\leq q-1$, where not all of the $i_s$ equal
to $0$. In fact, $i_0\neq0$, for otherwise the sequence $(a_1,\dots,
a_k)$ would not be $r$-independent. Moreover, $1-r^{i_0}$ is
invertible because by our assumption \eqref{prim} the image of $r$ in $\Bbb Z/d\Bbb Z$ is a
primitive $q$th root of 1 for every divisor $d$ of $n$.
We see that
$$j=(r^{i_1}a_1+\dots+r^{i_k}a_k-a_1-a_2-\dots-a_k)/(1-r^{i_0}),$$
so there are at most $q^{k+1}$ possibilities for $j$, as required.
\end{proof}

We will need a sufficient condition for a
sequence $(a_1,\dots,a_k)$ to contain an $r$-independent
subsequence.

\begin{lemma}\label{rigid}
Suppose that for some $m$ a sequence
$(a_1,\dots,a_k)$ of non-zero elements of $\Bbb Z/n\Bbb Z$
contains at least $q^m+m$ different values. Then
one can choose an $r$-independent subsequence
$(a_1,a_{i_2},\dots,a_{i_m})$ of $m$ elements that contains $a_1$.
\end{lemma}
\begin{proof}
If $m=1$, the lemma is obvious, as $a_1$ is $r$-independent by Remark~\ref{remark}.
So we assume that $m\geq 2$
and use induction on $m$. By induction we can choose an
$r$-independent subsequence of length $m-1$ starting from $a_1$.
Without loss of generality we can assume that $(a_1,a_{2},\dots,a_{m-1})$ is
an $r$-independent subsequence. By Lemma~\ref{115} there are at
most $q^m$ distinct elements in $D(a_1,a_{2},\dots,a_{m-1})$. Since there are at least
$q^m+m$ different values in the sequence $(a_1,\dots,a_k)$, the sequence
$a_m,a_{m+1}\dots,a_k$ contains at least $q^m+1$ different values.
So we can always choose an element $b\not\in D(a_1,a_{2},\dots,a_{m-1})$ among
$a_m,a_{m+1}\dots,a_k$ such that the sequence
$(a_1,a_{2},\dots,a_{m-1},b)$ is $r$-independent.
\end{proof}

\subsubsection*{`Selective nilpotency' condition on graded Lie algebras}

To recall the definitions of graded Lie algebras, homogeneous
elements and commutators, see the beginning of \S\,3.
Here we work with a $(\Bbb Z/n\Bbb Z)$-graded Lie ring $L$ such that $L_0=0$.
Formally the symbol $\Bbb Z/n\Bbb Z$ here means the
additive group $\Bbb Z/n\Bbb Z$. However, it will be convenient to
use the same symbol to denote also the ring $\Bbb Z/n\Bbb Z$. To
avoid overloaded notation we adopt the following convention.

\begin{Index Convention}
Henceforth a small Latin
letter with an index $i\in \Bbb Z/n\Bbb Z$ will denote a
homogeneous element in the grading component $L_i$, with the index
only indicating which component this element belongs to: $x_i\in
L_i$. We will not be using numbering indices for elements of the
$L_i$, so that different elements can be denoted by the same
symbol when it only matters which component the elements belong
to. For example, $x_{i}$ and $x_{i}$ can be different elements of
$L_{i}$, so that $[x_{i},\, x_{i}]$ can be a non-zero element of
$L_{2i}$.
\end{Index Convention}

Note that under Index Convention a homogeneous commutator
belongs to the component $L_s$ with index $s$ equal to the sum of
indices of all the elements involved in this commutator.
\vspace{1ex}

 \begin{definition}
 Let $n, q, r$ be integers defined by \eqref{prim}. We say that a $(\Z/n\Z)$-graded Lie ring $L$ satisfies
the \textit{selective $c$-nilpotency condition}
if,
under Index Convention,
\begin{equation}\label{select}
[x_{d_1},x_{d_2},\dots, x_{d_{c+1}}]=0\quad \text{ whenever }
(d_1,\dots,d_{c+1}) \text{ is $r$-independent}.
\end{equation}
 \end{definition}

\begin{notation}
We use the usual notation $(k,l)$ for the greatest common divisor of integers $k,l$.

Given $b\in\Bbb Z/n\Bbb Z$, we denote by $o(b)$
the order of $b$ (in the additive group). Thus, $o(b)$ is the least positive integer such that
$\underbrace{b+b+\dots+b}_{o(b)}=0$.
\end{notation}

In the following lemmas the condition $L_0=0$ ensures that all indices of non-trivial elements are non-zero, so
all sequences of such indices are either $r$-dependent or $r$-independent.

\begin{lemma}\label{lb}
Suppose that a $(\Z/n\Z)$-graded Lie ring $L$ with $L_0=0$ satisfies
the selective $c$-nilpotency
condition~\eqref{select},
and let $b$ be an element of $\Bbb{Z}/n\Bbb{Z}$ such that
$o(b)>2^{2^{2q-3}-1}c^{2^{2q-3}}$. Then there are at most
$q^{c+1}$ elements $a\in\Bbb{Z}/n\Bbb{Z}$ such that
$[L_a,\underbrace{L_b,\dots,L_b}_c]
\neq 0$.
\end{lemma}

\begin{proof}
Suppose that $[L_a,\underbrace{L_b,\dots,L_b}_{c}]
\neq0$. By \eqref{select} the sequence
$(a,\underbrace{b,\dots, b}_c)$ must be $r$-dependent and so we
have $a+b+\dots+b=r^{i_0}a+r^{i_1}b+\dots+r^{i_c}b$ for suitable
$0\leq i_j\leq q-1,$ where at least once $i_j\neq 0$. First suppose
that $i_0=0$. Then for some $1\leq m\leq c$ and some $1\leq
j_1,\dots,j_m\leq q-1$ we have $(r^{j_1}+\dots+r^{j_m}-m)b=0$,
where the $j_s$ are not necessarily different. Put
$m_0=r^{j_1}+\dots+r^{j_m}-m$ and $n_0=(n,m_0)$. Thus, $o(b)$
divides $n_0$ and $r^{j_1}+\dots+r^{j_m}-m=0$ in $\Bbb Z/n_0\Bbb
Z$. We collect terms re-writing $r^{j_1}+\dots+r^{j_m}-m$ as
$$b_{k_1}r^{k_1}+\dots+b_{k_l}r^{k_l}-m,$$
where $b_{k_i}>0$,
$\sum_{i=1}^l b_{k_i}=m$, and $q> k_1>k_2>\dots >k_l>0$ are all
different. By Lemma~\ref{char0} the polynomials
$X^{q-1}+\dots+X+1$ and $b_{k_1}X^{k_1}+\dots+b_{k_l}X^{k_l}-m$ have no common complex roots.
On the other hand, these polynomials
 have a common non-zero root in $\Bbb Z/n_0\Bbb Z$:
namely, the image of $r$. Therefore, by
Lemma~\ref{charp},
$$n_0\leq 2^{2^{2q-3}-1}m^{2^{2q-3}}\leq 2^{2^{2q-3}-1}c^{2^{2q-3}}.$$
This yields a contradiction since
$n_0\geq o(b)>2^{2^{2q-3}-1}c^{2^{2q-3}}$.

Hence, $i_0\neq0$. In this case $1-r^{i_0}$ is invertible by assumption \eqref{prim}.
Therefore,
$$a=(r^{i_1}b+\dots+r^{i_c}b-cb)/(1-r^{i_0}),$$
so there are at most
$q^{c+1}$ possibilities for $a$. The lemma follows.
\end{proof}

\begin{lemma}\label{l_b}
Suppose that a $(\Z/n\Z)$-graded Lie ring $L$ with $L_0=0$
satisfies the selective $c$-nilpotency condition~\eqref{select}.
There is a $(c,q)$-bounded number $w$ such that
$[L,\underbrace{L_b,\dots,L_b}_{w} ]=0$ whenever $b$ is an element
of $\Bbb{Z}/n\Bbb{Z}$ such that
$o(b)>\max\{2^{2^{2q-3}-1}c^{2^{2q-3}},q^{c+1} \}$.
\end{lemma}
\begin{proof}
Denote by $N(b)$ the set of all $a\in\Bbb{Z}/n\Bbb{Z}$
such that $[L_a,\underbrace{L_b,\dots,L_b}_c
]\neq0$. By
Lemma~\ref{lb}, \ $N=|N(b)|\leq q^{c+1}$. If for some $t\geq 1$ we
have $[L_a,\underbrace{L_b,\dots,L_b}_{c+t}
]\neq0$, then all
elements $a,a+b,a+2b,a+tb$ belong to $N(b)$. It follows that
either $[L_a,\underbrace{L_b,\dots,L_b}_{c+N}
]=0$ or $sb=0$ for
some $s\leq N$. But the latter case with $o(b)\leq N$ is impossible by
the hypothesis $o(b)>q^{c+1}\geq N$. Hence, $[L,\underbrace{L_b,\dots,L_b}_{c+q^{c+1}}
]=0$.
\end{proof}

Recall that for a given $r$-independent sequence
$(a_1,\dots,a_k)$ we denote by $D(a_1,\dots,a_k)$ the set
of all $j\in\Bbb Z/n\Bbb Z$ such that $(a_1,\dots,a_k,j)$ is
$r$-dependent.

Let $(d_1,\dots, d_c)$ be an arbitrary $r$-independent sequence, which we consider as fixed in the next few lemmas.

\begin{lemma}\label{odin}
Suppose that a $(\Z/n\Z)$-graded Lie ring $L$ with $L_0=0$ satisfies
the selective $c$-nilpotency
condition~\eqref{select},
 and let $U=[u_{d_1},\dots,u_{d_c}]$ be
a homogeneous commutator with the $r$-independent sequence of indices
$(d_1,\dots, d_c)$ (under Index Convention). Then every commutator of the form
\begin{equation}\label{eq4}
[U,x_{i_1},\dots,x_{i_t}]
\end{equation}
 can be
written as a linear combination of commutators of the form
\begin{equation}\label{eq5}
[U, m_{j_1},\dots,m_{j_{s}}],
\end{equation}
where $j_k\in D(d_1,\dots,d_c)$ and $s\leq t$. The case $s=t$ is
possible only if $i_k\in D(d_1,\dots, d_c)$ for all
$k=1,\dots,t$.
\end{lemma}

\begin{proof}
The assertion is obviously true for $t=0$. Let $t=1$.
If $i_1\in D(d_1,\dots,d_c)$, then $[U,x_{i_1}]$ is of the
required form. If $i_1\notin D(d_1,\dots,d_c)$, then
$[U,x_{i_1}]=0$ by \eqref{select} and there is nothing to prove.

Let us assume that $t>1$ and use induction on $t$. If all the
indices $i_j$ belong to $D(d_1,\dots,d_c)$, then the commutator
$[U,x_{i_1},\dots,x_{i_t}] $ is of the required form with $s=t$.
Suppose that in \eqref{eq4} there is an element $x_{i_k}$ with the index
$i_k$ that does not belong to $D(d_1,\dots,d_c)$. We choose such
an element with $k$ as small as possible and use $k$ as a second
induction parameter.

If $k=1$, then the commutator \eqref{eq4} is zero and we are done. Suppose that
$k\geq 2$ and write
$$[U,\dots,x_{i_{k-1}},x_{i_k},\dots,x_{i_t}]=[U,\dots,x_{i_k},x_{i_{k-1}},\dots,x_{i_t}]+
[U,\dots,[x_{i_{k-1}},x_{i_k}],\dots,x_{i_t}].$$
By the induction hypothesis the commutator
$[U,\dots,[x_{i_{k-1}},x_{i_k}],\dots,x_{i_t}]$ is a linear
combination of commutators of the form \eqref{eq5} because it is shorter
than \eqref{eq4}, while the commutator
$[U,\dots,x_{i_k},x_{i_{k-1}},\dots,x_{i_t}]$ is a linear
combination of commutators of the form \eqref{eq5} because the index that
does not belong to $D(d_1,\dots,d_c)$ here occurs closer to $U$
than in \eqref{eq4}.
\end{proof}

\begin{corollary}\label{tri}
Let $L$ and $U$ be as in Lemma~\ref{odin}. Then the ideal of $L$ generated by $U$ is
spanned by commutators of the form \eqref{eq5}.
\end{corollary}

In the next lemma we obtain a rather detailed information about
the ideal generated by $U$. Basically it says that  ${}_{\rm id}\langle U\rangle $ is generated
by commutators of the form \eqref{eq5} in which right after $U$ there are boundedly
many elements with indices in $D(d_1,\dots, d_c)$ followed only by elements with indices of small
additive order in $\Z /n \Z$.

\begin{notation}
From now on we fix the $(c,q)$-bounded number
$N(c,q)=\max\{2^{2^{2q-3}-1}c^{2^{2q-3}},q^{c+1}\}$, which appears
in Lemma~\ref{l_b}.

For our fixed $r$-independent sequence $(d_1,\dots, d_c)$,
we denote by $A$ the set of all $b\in D(d_1,\dots, d_c)$ such that
$o(b)> N(c,q)$ and by $B$ the
set of all $b\in D(d_1,\dots, d_c)$ such that
$o(b)\leq N(c,q)$. Let
$D=|D(d_1,\dots,d_c)|$ and let $w$
be the number given by Lemma~\ref{l_b}.
\end{notation}

\begin{lemma}\label{dva}
Let $L$ and $U$ be as in Lemma~\ref{odin}. The ideal of $L$ generated by $U$
 is spanned by
commutators of the form
\begin{equation}\label{eq6}
[U, m_{i_1},\dots,m_{i_u},m_{i_{u+1}},\dots,m_{i_{v}}],
\end{equation}
where $u\leq (w-1)D$, with $i_k\in D(d_1,\dots, d_c)$ for
$k\leq u$, and $i_k\in B$ for $k>u$.
\end{lemma}

\begin{proof}
Let $R$ be the span of all commutators of the form
\eqref{eq6}. It is sufficient to show that every commutator $W=[U,m_{j_1},\dots,m_{j_{s}}]$ of the form \eqref{eq5}
belongs to $R$. If $s\leq
(w-1)D$, it is clear that $W\in R$, so we assume that $s>(w-1)D$
and use induction on $s$. Write
\begin{equation}\label{eq61}
W= [U,m_{j_1},\dots,m_{j_t},
m_{j_{t-1}},\dots,m_{j_s}]+
[U,m_{j_1},\dots,[m_{j_{t-1}},m_{j_{t}}],\dots, m_{j_s}].
\end{equation}
If
$j_{t-1}+ j_{t}\in D(d_1,\dots,d_c)$, then the second summand is
of the form \eqref{eq5}.
Since it is shorter than $W$, it belongs to $R$
by the induction hypothesis. If $j_{t-1}+j_{t}\notin
D(d_1,\dots,d_c)$, then by Lemma~\ref{odin} the
second summand is a linear combination of
commutators of the form \eqref{eq5} each of which is shorter than $W$.

Thus, in either case the second summand in \eqref{eq61} belongs to
$R$. It follows that the commutator $W$ does not change modulo $R$
under any permutation of the $m_{j_k}$. If among the $m_{j_k}$
there are at least $w$ elements with the same index $j_k\in A$, we
move these elements next to each other. Then it follows from
Lemma~\ref{l_b} that $W=0$. Suppose now that all the indices
$j_k\in A$ occur less than $w$ times. We place all these elements
right after the $U$. This initial segment has length at most
$D(w-1)+c$, so the resulting commutator takes the required
form~\eqref{eq6}.
\end{proof}

\begin{corollary}\label{malocomp}
Suppose that a $(\Z/n\Z)$-graded Lie ring $L$ with $L_0=0$ satisfies
the selective $c$-nilpotency
condition~\eqref{select},
and let $(d_1,\dots, d_c)$ be an $r$-independent sequence.
Then the ideal
$_{\rm
id}\langle [L_{d_1},\dots, L_{d_c}]\rangle$ has
$(c,q)$-boundedly many non-trivial components of the induced
grading.
\end{corollary}

\begin{proof}
Let $U=[u_{d_1},\dots,u_{d_c}]$ be
an arbitrary homogeneous commutator with the given indices
(under Index Convention). Since in \eqref{eq6} we have ${i_k}\in D(d_1,\dots, d_c)$ for all
$k=1,\dots,u$, the sum of all indices of the initial segment
$[U, m_{i_1}, \dots, m_{i_u}]$ can take at most $D^{u}$
values. Denote by $Y$ the order of $\langle B\rangle$, the
subgroup of $\Bbb Z/n\Bbb Z$ generated by $B$. Clearly, the sum of the remaining indices in \eqref{eq6}
belongs to $\langle B\rangle$. It follows that the
sum of all indices in \eqref{eq6} can take at most $D^{u}Y$ values.
Since the order of every element in $B$ is $(c,q)$-bounded and
there are only $(c,q)$-boundedly many elements in $B$, it
follows that $Y$ is $(c,q)$-bounded. By Lemmas~\ref{115}, \ref{l_b}, and \ref{dva}
the number $u$ is also $(c,q)$-bounded. So
$_{\rm
id}\langle U\rangle$ has $(c,q)$-boundedly many non-trivial
components.

It follows from the proofs of Lemmas~\ref{odin} and \ref{dva}
that the set of indices of
all possible non-trivial
components in $_{\rm
id}\langle U\rangle$ is completely determined by
the tuple $(d_1,\dots,d_c)$ and does not depend on the choice of
$U=[u_{d_1},\dots,u_{d_c}]$. Since the ideal
$_{\rm
id}\langle[L_{d_1},\dots,L_{d_c}]\rangle$ is the sum of ideals
$_{\rm
id}\langle[u_{d_1},\dots,u_{d_c}]\rangle$ over all possible
$u_{d_1},\dots,u_{d_c}$, the result follows.
\end{proof}

\begin{lemma}\label{pyat}
Suppose that a homogeneous ideal $T$ of a Lie ring $L$
has only $e$ non-trivial components. Then $L$ has at most $e^2$
components that do not centralize $T$.
\end{lemma}
\begin{proof}
Let $T_{i_1},\dots,T_{i_e}$ be the non-trivial homogeneous
components of $T$ and let $S=\{i_1,\dots,i_e\}$. Suppose that
$L_i$ does not centralize $T$. Then for some $j\in S$ we have
$i+j\in S$. So there are at most $|S|\times|S|$ possibilities for
$i$, as required.
\end{proof}

\begin{proposition}\label{razresh}
Suppose that a $(\Z/n\Z)$-graded Lie ring $L$ with $L_0=0$ satisfies the selective $c$-nilpotency
condition~\eqref{select}. Then $L$ is solvable of $(c,q)$-bounded derived length $f(c,q)$.
\end{proposition}

Note that $L$ is solvable of $n$-bounded derived length by Kreknin's theorem, but we need a bound for the derived
length in terms of $c$ and $q$.

\begin{proof}
We use induction on $c$. If $c=0$, then $L=0$ and there is nothing to prove. Indeed,
$L_0=0$ by hypothesis, and for $d\ne 0$ we have
$L_d =0$ by \eqref{select}, since any element $d\ne 0$ is $r$-independent by Remark~\ref{remark}.

Now let $c\geq 1$. Let $I$ be the
ideal of $L$ generated by all commutators
$[L_{i_1},\dots,L_{i_c}]$, where $(i_1,\dots,i_c)$ ranges
through all $r$-independent sequences of length~$c$. The induced $(\Z/n\Z)$-grading
of $L/I$ has trivial zero-component and $L/I$ satisfies the selective
$(c-1)$-nilpotency condition.
By the induction hypothesis $L/I$ is
solvable of bounded derived length, say, $f_0$, that is,
$L^{(f_0)}\leq I$.

Consider an arbitrary
$r$-independent sequence $(i_1,\dots,i_c)$ and the ideal
$T={}_{\rm
id}\langle[L_{i_1},\dots,L_{i_c}]\rangle$. We know from
Corollary \ref{malocomp} that there are only $(c,q)$-boundedly
many, say, $e$, non-trivial grading components in $T$. By Lemma
\ref{pyat} there are at most $e^2$ components that do not
centralize $T$. Since $C_L(T)$ is also a homogeneous ideal, it
follows that the quotient $L/C_L(T)$ has at most $e^2$ non-trivial
components. Since the induced $(\Z/n\Z)$-grading of $L/C_L(T)$ also has
trivial zero component,
by Shalev's generalization~\cite{shalev} of Kreknin's theorem we conclude that
$L/C_L(T)$ is solvable of $e$-bounded derived length, say, $f_1$.
Therefore $L^{(f_1)}$, the corresponding term of the derived
series, centralizes $T$. Since $f_1$ does not depend on the choice
of the $r$-independent tuple $(i_1,\dots,i_c)$ and $I$ is the sum of all such ideals $T$,
it follows that $[L^{(f_1)},I]=0$. Recall that $L^{(f_0)}\leq I$. Hence,
$[L^{(f_1)},L^{(f_0)}]=0$. Thus $L$ is solvable
of $(c,q)$-bounded derived length at most $\max\{f_0,f_1\}+1$.
\end{proof}

\subsubsection*{Combinatorial corollary}
We now state a combinatorial corollary of Lemma~\ref{l_b}
and Proposition~\ref{razresh} that we shall need in the next
section for dealing with non-semisimple automorphisms, when eigenspaces do not form a direct sum.

We use the following notation:
$$\delta_1=[x_1, x_2], \qquad \delta_{k+1}=[\delta_k(x_1, \dots,
x_{2^k}),\,\, \delta_{k}(x_{2^k-1}, \dots, x_{2^{k+1}})].$$

\begin{corollary}\label{cor-razresh}
Let $n, q, r$ be positive integers such that
 $1\leq r \leq n-1$ and the image of $r$ in ${\Bbb Z}/d{\Bbb Z}$ is
a primitive $q$th root of $1$ for every divisor $d$ of $n$.

{\rm (a)} For  the function $f(c,q)$ given by Proposition~\ref{razresh},
the following holds. If we arbitrarily and formally assign
lower indices $i_1, i_2, \ldots
\in {\Bbb Z}$ to elements $y_{i_1}, y_{i_2}, \dots$
of an arbitrary  Lie ring, then the commutator
$\delta_{{f(c,q)}}(y_{i_1}, y_{i_2}, \dots y_{i_{2^{f(q,
c)}}})$ can be represented as a linear combination of commutators
in the same elements $y_{i_1}, y_{i_2}, \dots, y_{i_{2^{f(q,
c)}}}$ each of which contains either a subcommutator with zero
modulo $n$ sum of indices or a subcommutator of the form $[g
_{u_1}, g_{u_2},\dots ,g_{u_{c+1}} \,]$ with an $r$-independent
sequence $(u_1,\dots,u_{c+1})$ of indices, where elements $ g_j$ are
commutators in $y_{i_1}, y_{i_2}, \dots, y_{i_{2^{f(c,q)}}}$
such that the sum of indices of all the elements involved in $g_j$ is congruent to
$j$ modulo~$n$.

{\rm (b)} For the function $w=w(c,q)$ given by Lemma~\ref{l_b} any
commutator $[y_{a}, \underbrace{x_{b}, y_b, \dots ,z_{b}}_{w}]$,
with $w$ elements with the same index $b$ over the brace, can be
represented as a linear combination of the same form as in (a)
whenever the image of $b$ in $\Bbb{Z}/n\Bbb{Z}$ is such that
$o(b)>N(c,q)=\max\{2^{2^{2q-3}-1}c^{2^{2q-3}},q^{c+1} \}$.
\end{corollary}

\begin {proof}
(a)  Let $M$ be a free Lie ring freely generated
by $y_{i_1}, y_{i_2}, \dots, y_{i_{2^{f(c,q)}}}$. We define for
each $i=0,\,1,\,\dots ,n-1$ the additive subgroup $M_i$ of $M$
generated (in the additive group) by all commutators in the
generators $y_{i_j}$ with the sum of indices congruent to $i$
modulo $n$. Then, obviously, $M=M_0\oplus M_1\oplus \cdots \oplus
M_{n-1}$ and $ [M_i,M_j]\subseteq M_{i+j\,({\rm mod\, n)}}$, so
that this is a $({\Bbb Z} /n{\Bbb Z})$-grading.
 By Proposition~\ref{razresh} we obtain
 $$\delta_{{f(c,q)}}(y_{i_1}, y_{i_2}, \dots, y_{i_{2^{f(c,q)}}})\in {}_{{\rm
id}}\!\left<M_0\right>+\sum_{\substack{(u_1,\dots,u_{c+1})\\ \;\text{is $r$-independent}}}
{}_{{\rm id}}\!\left< [M_{u_1}, M_{u_2},\dots
,M_{u_{c+1}}] \right>.$$ By the definition of $M_i$ this inclusion
is equivalent to the required equality in the conclusion of
Corollary \ref{cor-razresh}(a). Since the elements $y_{i_1}, y_{i_2},
\dots y_{i_{2^{f(c,q)}}}$ freely generate the Lie ring $M$, the
same equality holds in any Lie ring.

(b) The proof of the second statement is obtained by
the same arguments with the only difference that
Lemma~\ref{l_b} is applied instead of Proposition~\ref{razresh}.
\end{proof}

\section{Bounding the nilpotency class}

In this section we obtain bounds for the nilpotency class of
groups and Lie rings admitting a metacyclic Frobenius group of
automorphisms with fixed-point free kernel. Earlier such results were
obtained by the second and third authors \cite{mak-shu10} in the case
of the kernel of prime order. Examples at the end of the section
show that such nilpotency results are no longer true for non-metacyclic
Frobenius groups of automorphisms. The results for groups are
consequences of the corresponding results for Lie rings and
algebras, by various Lie ring methods.

\subsubsection*{Lie algebras} We begin with the theorem
for Lie algebras, which is devoid of technical details
required for arbitrary Lie rings, so that the ideas of proof are more clear.
Recall that $C_L(A)$ denotes the fixed-point subalgebra for a group of
automorphisms $A$ of a Lie algebra~$L$, and that nilpotency class
is also known as nilpotency index.

\begin{theorem}\label{liealg}
Let $FH$ be a Frobenius group with cyclic kernel
$F$ of order $n$ and complement $H$ of order $q$. Suppose that
$FH$ acts by automorphisms on a Lie algebra $L$ in such a way that
$C_L(F)=0$ and $C_L(H)$ is nilpotent of class $c$.
Then $L$ is nilpotent of $(c,q)$-bounded class.
\end{theorem}

Note that the assumption of the kernel $F$ being cyclic is
essential: see the corresponding examples at the end of the section.

\begin{proof}
Our first aim is to reduce the problem to the study of ${\Bbb Z}/n {\Bbb
Z}$-graded Lie algebras satisfying the selective $c$-nilpotency
condition~\eqref{select}. Then the solubility of $L$ of
$(c,q)$-bounded derived length will immediately follow from
Proposition~\ref{razresh}. Further arguments are then applied to obtain nilpotency.

Let $\omega$ be a primitive $n$th root of $1$. We extend the
ground field by $\omega$ and denote the resulting Lie algebra by $\widetilde L$.
The group $FH$ acts in a natural way on $\widetilde L$ and this action
inherits the conditions that $C_{\widetilde
L}(F)=0$ and $C_{\widetilde L}(H)$ is nilpotent of class~$c$. Thus,
we can assume that $L=\widetilde L$ and the ground field contains~$\omega$, a~primitive $n$th root
 of~$1$.

Let $\varphi$ be a generator of $F$. For each $i=0,\dots,n-1$ we
denote by $L_i=\{x\in L\mid x^{\varphi}=\omega^ix\}$ the eigenspace for the eigenvalue $\omega^i$. Then
$$[L_i, L_j]\subseteq L_{i+j\,(\rm{mod}\,n)}\qquad \text{and}\qquad L= \bigoplus _{i=0}^{n-1}L_i,$$
so this is a $(\Z /n\Z
)$-grading. We also have $L_0=C_L(F)=0$.

 Since $F$ is cyclic of order $n$,
$H$ is also cyclic. Let $H= \langle h \rangle$ and let $\varphi^{h^{-1}}
= \varphi^{r}$ for some $1\leq r \leq n-1$. Then $r$ is a
primitive $q$th root of $1$ in ${\Bbb Z}/n {\Bbb Z}$, and,
moreover, the image of $r$ in ${\Bbb Z}/d {\Bbb Z}$ is a primitive
$q$th root of $1$ for every divisor $d$ of $n$, since $h$ acts fixed-point-freely
on every subgroup of $F$. Thus, $n,q,r$ satisfy condition \eqref{prim}.

We shall need a simple remark that we can  assume
that  $n=|F|$ is  not divisible by the characteristic of the
ground field. Indeed, if the characteristic $p$ divides $n$, then
the Hall $p'$-subgroup $\langle \chi\rangle$ of $F$ acts
fixed-point-freely on $L$ --- otherwise the Sylow $p$-subgroup $\langle \psi\rangle$ of $F$
would have non-trivial fixed points on the $\psi$-invariant subspace $C_L(\chi )$,
and these would be non-trivial fixed points for $F$.
Thus, $L$ admits the Frobenius group
of automorphisms $\langle\chi\rangle H$ with $C_L(\chi)=0$.
Replacing $F$ by $\langle \chi\rangle$ we can assume that $p$ does
not divide $n$.

The group $H$ permutes the components $L_i$ so that ${L_i}^h =
L_{ri}$ for all $i\in \Bbb Z/n\Bbb Z$. Indeed, if $x_i\in L_i$,
then $(x_i^{h})^{\varphi} = x_i^{h\varphi h^{-1}h} =
(x_i^{\varphi^{r}})^h = \omega^{ir}x_i^h$.

Given $u_k\in
L_k$, we temporarily denote $u_k^{h^i}$ by $u_{r^ik}$ (under Index Convention). The sum over any
$H$-orbit belongs to $C_L(H)$ and therefore
$u_k+u_{rk}+\cdots+u_{r^{q-1}k}\in C_L(H)$. Let
$x_{a_1},\dots,x_{a_{c+1}}$ be homogeneous elements in
$L_{a_1},\dots,L_{a_{c+1}}$, respectively. Consider the elements
\begin{align*}
X_1&=x_{a_1}+x_{ra_1}+\cdots+x_{r^{q-1}a_1},\\
\vdots&\\
X_{c+1}&=x_{a_{c+1}}+x_{ra_{c+1}}+\cdots+x_{r^{q-1}a_{c+1}}.
\end{align*}
Since all of them lie in $C_L(H)$, which is nilpotent
of class $c$, it follows that
$$[X_1,\dots,X_{c+1}]=0.$$
After
expanding the expressions for the $X_i$ we obtain on the left a linear combination of
commutators in the $x_{r^ja_i}$, which in particular involves the term
$[x_{a_1},\dots,x_{a_{c+1}}]$. Suppose that the commutator
$[x_{a_1},\dots,x_{a_{c+1}}]$ is non-zero. Then there must be other terms in the expanded
expression that belong to the same component
$L_{a_1+\cdots+a_{c+1}}$. In other words, then
$$a_{1}+\dots+a_{c+1}=r^{\alpha_1}a_{1}+\dots+r^{\alpha_{c+1}}a_{c+1}$$
for some $\alpha_i\in\{0,1,2,\dots,q-1\}$ not all of which are zeros, so that
$(a_1,\dots, a_{c+1})$ is an $r$-dependent sequence.
This means that $L$ satisfies the selective $c$-nilpotency condition~\eqref{select}.
By Proposition~\ref{razresh} we obtain that $L$ is solvable of $(c,q)$-bounded derived length.

We now use induction on the derived length to prove that $L$ is
nilpotent of $(c,q)$-bounded class. If $L$ is abelian, there is
nothing to prove. Assume that $L$ is metabelian --- this is, in
fact, the main part of the proof. When $L$ is metabelian,
$[x,y,z]=[x,z,y]$ for every $x\in [L,L]$ and $y,z\in L$.
The key step is a stronger version of Lemma~\ref{l_b} for the
metabelian case, without the arithmetical condition on the
additive orders.

\begin{lemma}\label{l_b-metab} Let $L$ be metabelian.
There is a $(c,q)$-bounded number $m$ such that
{$[L,\underbrace{L_b,\dots,L_b}_{m}
]=0$} for every $b\in \Bbb{Z}/n\Bbb{Z}$.
\end{lemma}

\begin{proof}
For each $a\in \Bbb Z/n\Bbb Z$ we denote $[L,L]\cap L_a$ by $[L,L]_a$.
Clearly, it suffices to show that $[[L,L]_a,\underbrace{L_b,\dots,L_b}_{m-1}
]=0$ for every $a,b\in \Bbb{Z}/n\Bbb{Z}$; we can of course assume
that $a,b\ne 0$. First suppose that $(n,b)=1$. If $n$ is large
enough, $n>N(c,q)$, then the order $o(b)=n$ is also large enough
and the result follows by Lemma~\ref{l_b} with  $m-1=w$, where $w$
is given by Lemma~\ref{l_b}. If, however, $n\leq N(c,q)$, then we
find a positive integer $k<n$ such that $a+kb=0$ in
$\Bbb{Z}/n\Bbb{Z}$. (As usual, we
freely switch from considering positive integers to their images
in $\Bbb{Z}/n\Bbb{Z}$ without changing notation.) Then
$[L_a,\underbrace{L_b,\dots,L_b}_{k} ]\subseteq L_0=0$, so that we can put $m-1=N(c,q)$.

Now suppose that $(n,b)\ne 1$, and let $s$ be a prime dividing
both $n$ and $b$. If $s$ also divides $a$, then the result follows
by induction on $|F|$ applied to the Lie algebra $C_L(\f
^{n/s})=\sum _iL_{si}$ containing both $L_a$ and $L_b$: this Lie
algebra is $FH$-invariant and $\f $ acts on it as an
automorphism of order $n/s$ without non-trivial fixed points. The
basis of this induction is the case of $n=|F|$ being a prime,
where necessarily $(n,b)=1$, which has already been dealt with above.
(Alternatively, we can refer to the main result of
\cite{mak-shu10}, where Theorem~\ref{liealg} was proved for $F$ of prime order.)
Note that the function $m=m(c,q)$
remains the same in the induction step, so no dependence on $|F|$ arises.

Thus, it remains to consider the case where the prime $s$ divides
both $b$ and $n$ and does not divide $a$.
Using the same notation $x_{r^ik} =x_k^{h^i}$ (under Index Convention), we
have
$$\Big[u_a+u_{ra}+\dots +u_{r^{q-1}a},\,\underbrace{
(v_b+v_{rb}+\dots +v_{r^{q-1}b}),\dots , (w_b+w_{rb}+\dots +w_{r^{q-1}b})}_{c}\Big]=0$$
for any $c$ elements $v_b,\dots ,w_b \in L_b$, because here sums are  elements of $C_L(H)$.

By \eqref{prim} any two indices $r^ia,\,r^ja$ here are different modulo $s$, while all the indices
above the brace are divisible by $s$, which divides $n$.
Hence we also have
\begin{equation}\label{mod-s1}
\Big[u_a,\, \underbrace{ (v_b+v_{rb}+\dots
+v_{r^{q-1}b}),\dots , (w_b+w_{rb}+\dots +w_{r^{q-1}b})}_{c}\,\Big]=0.
\end{equation}
Let $Z$ denote the span of all the sums $x_b+x_{rb}+\dots
+x_{r^{q-1}b}$ over $x_b\in L_b$ (in fact, $Z$ is the fixed point subspace of $H$ on
$\bigoplus _{i=0}^{q-1}L_{r^ib}$). Then \eqref{mod-s1}
means that
$$\Big[L_a,\, \underbrace{Z,\dots ,
Z}_{c}\,\Big]=0.$$
 Applying $\f ^j$ we also obtain
\begin{equation}\label{mod-s2}
\Big[L_a,\,
\underbrace{Z^{\f ^j},\dots , Z^{\f
^j}}_{c}\,\Big]=\Big[L_a^{\f
^j},\, \underbrace{Z^{\f ^j},\dots , Z^{\f
^j}}_{c}\,\Big]=0.\end{equation}

A Vandermonde-type linear algebra argument shows that $L_b\subseteq \sum _{j=0}^{q-1}Z^{\f
^j}$. We prove this fact in a greater generality, which will be needed later.

\begin{lemma}\label{vander}
Let $\langle \f\rangle$ be a cyclic group of order $n$, and $\w$ a primitive $n$th root of unity.
Suppose that $M$ is a $\Z [\w ]\langle \f\rangle$-module such that $M=\sum _{i=1}^mM_{k_i}$,
where $x\f =\w ^{k_i}x$ for  $x\in M_{k_i}$ and  $0\leq k_1<k_2<\cdots
<k_m<n$. If $z=y_{k_1}+y_{k_2}+\cdots +y_{k_m}$ for $y_{k_i}\in
M_{k_i}$, then for some $m$-bounded number $l_0$  every element $n^{l_0}y_{k_s}$
is a $\Z [\w ]$-linear combination of the elements
$z,\,z\f,\dots ,z\f ^{m-1}$.
\end{lemma}

\begin{proof}
We have
$$z{\f ^j}= \w ^{jk_1} y_{k_1}+\w ^{jk_2} y_{k_2}+\dots +\w
^{jk_m}y_{k_m}.$$
Giving values $j=0,\dots ,m-1$ we
obtain $m$ linear combinations of the elements $y_{k_1}, y_{k_2}, \dots ,y_{k_m}$.
Then for every $s=1,\dots ,m$
a suitable linear combination of these linear combinations produces $Dy_{k_s}$, where
$D$ is the Vandermonde determinant of the $m\times m$ matrix of
coefficients of these linear combinations, which is equal to
$$\prod\limits_{1\leq i<j\leq
m}(\w  ^{k_{i}}-\w  ^{k_j}).$$
Each factor is a product of an
(invertible) power of $\w$ and an element
$1-\w  ^{k_i-k_j}$. By hypothesis the exponent $k_i-k_j$
is not divisible by $n$, so
$\rho= \w  ^{k_i-k_j}$ is a root of $1+x+\dots
+x^{n-1}=(x-\rho )g(x)$. By substituting $x=1$ we obtain
$n=(1-\rho )g(1)$, so by multiplying $Dy_{k_s}$ by an invertible power of $\w$ and
several elements of the type $g(1)$ we
obtain $n^{l_0}y_{k_s}$ for some $m$-bounded number $l_0$. As a result, $n^{l_0}y_{k_s}$ is expressed as a
$\Z [\w ]$-linear combination of the elements
$z,\,z\f,\dots ,z\f ^{m-1}$.
\end{proof}

In the proof of Lemma~\ref{l_b-metab} we can apply Lemma~\ref{vander} with  $M=L_b+L_{rb}+\dots
+L_{r^{q-1}b}$ and $m=q$ to $w=v_b+v_{rb}+\dots
+v_{r^{q-1}b}\in Z$ for any $v_b\in L_b$, because here all the indices $r^ib$ can be regarded as
pairwise distinct residues modulo $n$ by condition~\eqref{prim} (that $r$ is a primitive $q$th
root of unity modulo any divisor of $n$). Using the fact that $n$ is invertible in our ground field
we obtain that $L_b\subseteq \sum _{j=0}^{q-1}Z^{\f ^j}$.

We now claim that $$[[L,L]_a,\, \underbrace{L_b,\dots ,
L_b}_{q(c-1)+1}]=0.$$ Indeed, after replacing $L_b$ with $\sum
_{j=0}^{q-1}Z^{\f ^j}$ and expanding the sums, in each
commutator of the resulting linear combination  we can freely
permute the entries $Z^{\f ^j}$, since $L$ is  metabelian. Since there are sufficiently many of them, we
can place  at least $c$ of the same
$Z^{\f ^{j_0}}$ for some $j_0$ right after $[L,L]_a$ at the beginning,
which gives $0$ by \eqref{mod-s2}.
\end{proof}

We now return to  the case of metabelian $L$ in the proof of Theorem~\ref{liealg}.
Let $m=m(c,q)$ be as in Lemma~\ref{l_b-metab}
and put $g=(m-1)(q^{c+1}+c)+2$. For any sequence of $g$ non-zero elements
$(a_1,\dots,a_g)$ in $\Bbb Z/n\Bbb Z$ consider the commutator
$[[L,L]_{a_1},L_{a_2},\dots,L_{a_g}]$.
If the sequence
$(a_1,\dots,a_g)$ contains an $r$-independent subsequence of
length $c+1$ that starts with $a_1$, by
 permuting the $L_{a_i}$ we can assume that $a_1,\dots,a_{c+1}$ is the
$r$-independent subsequence. Then  $[[L,L]_{a_1},L_{a_2},\dots,L_{a_g}]=0$ by~\eqref{select}.
If the sequence
$(a_1,\dots,a_g)$ does not contain an $r$-independent subsequence of
length $c+1$ starting with $a_1$, then by Lemma~\ref{rigid} the sequence $(a_1,\dots,a_g)$
contains at most $q^{c+1}+c$ different values. The number $g$ was
chosen big enough to guarantee that either the value of $a_1$
occurs in $(a_1,\dots,a_g)$ at least $m+1$ times or, else, another
value, different from $a_1$, occurs at least $m$ times. Using that
$[x,y,z]=[x,z,y]$ for $x\in [L,L]$ we can assume that $a_2=\dots=a_{m+1}$, in which
case it follows from Lemma~\ref{l_b-metab} that
$[[L,L]_{a_1},L_{a_2},\dots,L_{a_g}]=0$. Thus, we conclude that
$L$ is nilpotent of class at most $g$.

Now suppose that the derived length of $L$ is at least 3. By the
induction hypothesis, $[L,L]$ is nilpotent of bounded class.
According to the previous paragraph, the quotient
$L/[[L,L],[L,L]]$ is nilpotent of bounded class, as well.
Together, this gives nilpotency of $L$ of bounded class by the Lie
ring analogue \eqref{chao} of P.~Hall's theorem.
\end{proof}

\subsubsection*{Lie groups and torsion-free locally nilpotent groups}

We now derive the group-theoretic consequences of
Theorem~\ref{liealg}; but the theorem on finite groups will have
to wait until we prove a similar result for Lie rings.
By the
well-known connection between Lie groups and their Lie algebras,
the following theorem is an immediate consequence of
Theorem~\ref{liealg}. Recall that $C_G(A)$ denotes the
fixed-point subgroup for a group of automorphisms
$A$ of a group $G$, and that nilpotency class
is also known as nilpotency index.

\begin{theorem}\label{gr-lie}
Suppose that a connected
Lie group $G$ (complex or real)
admits a finite Frobenius group of automorphisms $FH$ with
cyclic kernel $F$ of order $n$ and complement $H$ of
order $q$ such that $C_G(F)=1$ and $C_G(H)$ is nilpotent of
class $c$. Then $G$ is nilpotent
of $(c,q)$-bounded class.
\end{theorem}

\begin{proof}
Every
automorphism $\alpha$ of $G$ induces the automorphism $d_e\alpha $
of the tangent Lie algebra $\frak g$ of $G$ which is the
differential of $\alpha$ at identity. The fixed-point subalgebra
$C_{\frak g}(d_e\alpha )$ is the tangent Lie algebra of the
fixed-point subgroup $C_G(\alpha )$ (see, for example,
\cite[Theorem~3.7]{oni-gor}). Therefore the group of
automorphisms $\overline{FH}$ of $\frak g$ induced by $FH$ has the
properties that $C_{\frak g}(\overline F)=0$ and $C_{\frak
g}(\overline H)$ is nilpotent of class~$c$. (This can also be
shown by using the $\mathsf{Exp}$ and $\mathsf{Log}$ functors, which locally
commute with automorphisms of~$G$.)
By
Theorem~\ref{liealg}, the Lie algebra ${\frak g}$ is nilpotent
of $(c,q)$-bounded class.  Since $G$ is connected,
this implies the same result for $G$.
\end{proof}

Another corollary of Theorem~\ref{liealg}  follows by  a similar Lie ring method ---
the Mal'cev correspondence, which is
also based on the $\mathsf{Exp}$ and
$\mathsf{Log}$ functors and the Baker--Campbell--Hausdorff formula.

\begin{theorem}\label{torsion-free}
Suppose that a
locally nilpotent torsion-free group $G$
admits a finite Frobenius group of automorphisms $FH$ with
cyclic kernel $F$ of order $n$ and complement $H$ of
order $q$ such that $C_G(F)=1$ and $C_G(H)$ is nilpotent of
class $c$. Then $G$ is nilpotent
of $(c,q)$-bounded class.
\end{theorem}

\begin{proof} It is clearly sufficient to prove the theorem for finitely generated $FH$-invariant
subgroups of $G$, so we can assume that $G$ is finitely generated and therefore nilpotent.
Let $\widehat G$ be the \textit{Mal'cev completion} of $G$ obtained by
adjoining all roots of non-trivial elements of $G$ (see
\cite{mal49} or, for example, \cite[\S\,10.1]{khubook2}).
Forming the completion preserves the
nilpotency class of a nilpotent subgroup. Every automorphism
$\alpha$ of $G$ can be canonically extended to an automorphism of
$\widehat G$, which we denote by the same letter. The fixed point
subgroup $C_{\widehat G}(\alpha )$ is the completion of $C_{
G}(\alpha )$. Applying this to $FH$, we see that as a group of automorphisms of $\widehat
G$, it inherits the properties that $C_{\widehat G}(F)=1$ and $C_{\widehat
G}(H)$ is nilpotent of class $c$.

Under the Mal'cev correspondence, the radicable locally nilpotent
torsion-free group $\widehat G$ can be viewed as a locally
nilpotent Lie algebra $L$ (with the same underlying set) over the
field of rational numbers $\Q$, with the Lie ring operations given
by the inversions of the Baker--Campbell--Hausdorff formula. The
automorphisms of $\widehat G$ are automorphisms of $L$ acting on
the same set in the same way, and we denote them by the same
letters. Thus, $C_{L}(F)=0$ and $C_{L}(H)$ is nilpotent of class
$c$. By Theorem~\ref{liealg} the Lie algebra $L$ is nilpotent
of $(c,q)$-bounded class. Hence the
same is true for $\widehat G$ and therefore also for $G$.
\end{proof}

\subsubsection*{Lie rings}
We now prove a similar theorem for arbitrary Lie rings. The
additional conditions on the additive group of the Lie ring will be
automatically satisfied when the theorem is
later used in the proof of the main result on the nilpotency class of a finite group
with a metacyclic Frobenius group of automorphisms. We also include
the solvability result, which does not require those additional conditions.

\begin{theorem}\label{liering}
Let $FH$ be a Frobenius group with cyclic kernel
$F$ of order $n$ and complement $H$ of order $q$. Suppose that
$FH$ acts by automorphisms on a Lie ring $L$ in such a way that
$C_L(F)=0$ and $C_L(H)$ is nilpotent of class $c$. Then

{\rm (a)} the Lie ring  $L$ is solvable of $(c,q)$-derived length;

{\rm (b)} for some functions $u=u(c,q)$ and $v=v(c,q)$ depending only on
$c$ and $q$, the Lie subring $n^{u}L$ is nilpotent of class $v$, that is,
$\gamma_{v+1}(n^{u}L)=n^{u(v+1)}\g _{v+1}(L)=0$;

{\rm (c)} moreover, $L$ is nilpotent of $(c,q)$-bounded class in either of the following cases:
\vspace{-1.5ex}
 \begin{enumerate}\itemsep-1ex\itemindent1.7em
\item[{\rm (i)}] $L$ is a Lie algebra over a field;
\item[{\rm (ii)}] the additive group of $L$ is periodic (includes the case of $L$ finite);
\item[{\rm (iii)}] $n$ is invertible in the ground ring of $L$;
\item[{\rm (iv)}] $nL=L$;
\item[{\rm (v)}] $n$ is a prime-power.
\end{enumerate}
\end{theorem}

We included for completeness the case of algebras over a field,
which is Theorem~\ref{liealg}. It is not clear
at the moment if the additional conditions (i)--(v) are really
necessary. However, many important cases are covered,
including those needed in our group-theoretic applications.
The solvability result is included precisely because it
does not require these conditions: of course,  $L$ is solvable of
$n$-bounded derived length by Kreknin's theorem, but we obtain a
bound for the derived length in terms of $|H|$ and the nilpotency
class of $C_L(H)$.

\begin{proof}
The proof is basically along the same lines as the proof of
Theorem~\ref{liealg} for algebras; the complications arise
because we no longer have a direct sum of eigenspaces forming a
$({\Bbb Z}/n {\Bbb Z})$-grading.

We extend the
ground ring by a primitive $n$th root of unity $\omega$ setting $\widetilde L=L\otimes_{\Bbb
Z}\Bbb Z[\omega]$. The group $FH$ acts in the natural way on $\widetilde L$
and the action inherits the conditions that $C_{\widetilde
L}(F)=0$
and $C_{\widetilde L}(H)$ is nilpotent of class~$c$. Since the other conditions (i)--(v) would also hold
for $\widetilde L$, and the conclusion of the theorem for $L$ would follow from the same conclusion for $\widetilde L$,
we can assume that $L=\widetilde L$ and the ground ring contains~$\omega$, a~primitive $n$th root
 of~$1$.

Let $F=\langle\varphi\rangle$. For each $i=0,\dots,n-1$ we
define the `eigenspace' for $\omega^i$ as $L_i=\{x\in L\mid x^{\varphi}=\omega^ix\}$.
Then
$$[L_i, L_j]\subseteq L_{i+j\,(\rm{mod}\,n)}\qquad \text{and}\qquad nL\subseteq \sum _{i=0}^{n-1}L_i.$$
This is `almost a $(\Z /n\Z
)$-grading' --- albeit of $nL$ rather than of $L$, and although the sum is not
necessarily direct, any linear dependence of elements from different $L_i$ is annihilated by $n$:
\begin{equation}\label{almost-ind}
\text{if}\quad l_1+l_2+\dots +l_{n-1}=0,\quad\text{then}\quad nl_1=nl_2=\dots =nl_{n-1}=0
\end{equation}
(see, for example,
\cite[Lemma 4.1.1]{khu}).  We also have $L_0=C_L(F)=0$.

We shall mostly work with the $FH$-invariant subring $K=\sum_{i=0}^{n-1}L_i$.

As in the proof of Theorem~\ref{liealg}, let $H= \langle h \rangle$ and $\varphi^{h^{-1}}
= \varphi^{r}$ for some $1\leq r \leq n-1$, so that  $n,q,r$ satisfy \eqref{prim} and
${L_i}^h = L_{ri}$ for all $i\in \Bbb Z/n\Bbb Z$.

Using the same notation $u_{r^ik}=u_k^{h^i}$ and
$X_j=x_{a_j}+x_{ra_j}+\cdots+x_{r^{q-1}a_j}\in C_L(H)$, we have
$$[X_1,\dots,X_{c+1}]=0.$$
Expand all the $X_i$ and consider the resulting linear combination of commutators in the $x_{r^ja_i}$.
Suppose that $[x_{a_1},\dots,x_{a_{c+1}}]\ne 0$. In the `semisimple' case,
the sum $K=\sum_{i=0}^{n-1}L_i$ is direct and
there would have to be other terms in the same component
$L_{a_1+\cdots+a_{c+1}}$, which implies that
$$a_{1}+\dots+a_{c+1}=r^{\alpha_1}a_{1}+\dots+r^{\alpha_{c+1}}a_{c+1}$$
for some $\alpha_i\in\{0,1,2,\dots,q-1\}$ not all of which are zeros, so that
$(a_1,\dots, a_{c+1})$ is an $r$-dependent sequence.
This would mean that $K=\sum_{i=0}^{n-1}L_i$ satisfies the selective $c$-nilpotency condition
\eqref{select}.

In the general case, where $K$ may not be the direct sum of the $L_i$, it can
happen that there are no other terms in the expanded
expression with the same sum of indices
${a_1+\cdots+a_{c+1}}$. But then the `almost linear independence' \eqref{almost-ind}  implies that
$n[x_{a_1},\dots,x_{a_{c+1}}]=0$. Thus, in any case,
\begin{equation}\label{if-r-ind}
n[x_{d_1},x_{d_2},\dots,
x_{d_{c+1}}]=0\qquad \text{ whenever } (d_1,\dots,d_{c+1}) \text{ is
$r$-independent}.
\end{equation}

We are now ready to prove  part (a) on solvability.
Let $y_{i_j}\in L_{i_j}$ be any elements of the `eigenspaces' (under Index Convention).
By Corollary~\ref{cor-razresh}(a) the commutator $\delta_{{f(c,q)}}(y_{i_1}, y_{i_2},
\dots, y_{i_{2^{f(c,q)}}})$ can be represented as a linear
combination of commutators
each of which contains either a
subcommutator with zero modulo $n$ sum of indices or a
subcommutator $[g _{d_1}, g_{d_2},\dots ,g_{d_{c+1}}
\,]$ with $r$-independent sequence $(d_1,\dots,d_{c+1})$ of indices.
By hypothesis, $L_0=0$, and
$n[g _{d_1}, g_{d_2},\dots ,g_{d_{c+1}} \,]=0$ by property~\eqref{if-r-ind}.
Hence, $n\delta_{{f(c,q)}}(y_{i_1}, y_{i_2}, \dots, y_{i_{2^{f(q,
c)}}})=0$.
Thus, $nK^{(f(c,q))}=0$ and therefore
$$n(nL)^{f(c,q))}=n^{2^{f(c,q)}+1}L^{(f(c,q))}=0. $$
In particular,  the additive group of $T=L^{(f(c,q))}$ is a
periodic abelian group. We
decompose it into the direct sum of
Sylow subgroups
$$T=T_{p_1}\oplus T_{p_2}\oplus \cdots \oplus T_{p_r},$$
where $p_1, p_2,\dots, p_r$ are the prime divisors of $n$. The
$T_{p_k}$ are $FH$-invariant ideals and
$[T_{p_i}, T_{p_j}]=0$ for $i\neq j$.

Let $p\in \{p_1, p_2, \dots, p_r\}$ and $n=p^ks,$ where
$(p,s)=1$. Let $ \langle \psi \rangle$ be the Sylow $p$-subgroup
of  $F=\langle \varphi\rangle$ and let $\langle
\varphi\rangle =\langle \psi\rangle \times \langle \chi\rangle$,
where $s=|\chi |$.
 The fixed-point subring $C=C_{T_{p}}(\chi )$ is a
$\psi$-inva\-riant abelian $p$-group and $ C_C(\psi )\subseteq
C_L(\varphi )=0$. The automorphism $\psi$
of order $p^t$ acting on an abelian $p$-group necessarily has
non-trivial fixed points. Hence, $C_{T_{p}}(\chi )=0$. Thus the
subring $T_p$ admits the Frobenius group of automorphisms
$\langle\chi\rangle H$ with cyclic kernel $\langle\chi\rangle$ and
complement $H$ of order $q$ such that $C_{T_p}(\chi)=0$ and
$C_{T_p}(H)$ is nilpotent of class at most $c$. By the above argument,
$$0=s(sT_p)^{(f(c,q))}=s^{2^{f(c,q)}+1}T_p^ {(f(c,q))},$$
whence $T_p^{(f(c,q))}=0$ for each prime $p$. Hence $T^{(f(c,q))}=(L^{(f(c,q))})^{(f(c,q))}=0$ and
$L$ is solvable of derived length at most $2f(c,q)$. The proof of part (a) is complete.
\medskip

In the proof of the nilpotency statements, we need an analogue of Lemma~\ref{l_b-metab}.

\begin{lemma}\label{l_b-metab2}
There are  $(c,q)$-bounded numbers $l$ and $m$ such that for every $b\in \Bbb{Z}/n\Bbb{Z}$
$$n^{l}[K,\underbrace{L_b,\dots,L_b}_{m}
]\subseteq \break [[K,K],\,[K,K]].$$
\end{lemma}

\begin{proof}
We denote $[K,K]\cap L_a$ by
$[K,K]_a$. Clearly, it suffices to show that
$$n^l[[K,K]_a,\underbrace{L_b,\dots,L_b}_{m-1}
]\subseteq [[K,K],\,[K,K]]$$ for every $a,b\in \Bbb{Z}/n\Bbb{Z}$; we can of course assume
that $a,b\ne 0$. First suppose that $(n,b)=1$. If $n$ is large
enough, $n>N(c,q)$, then the order $o(b)=n$ is also large enough. Then by  Corollary~\ref{cor-razresh}(b),
for the corresponding function $w=w(c,q)$,
any commutator $[y_{a}, \underbrace{x_{b}, y_b, \dots ,z_{b}}_{w}]$
can be represented as a linear combination of commutators
each of which has either a subcommutator with zero
modulo $n$ sum of indices or a subcommutator $[g
_{u_1}, g_{u_2},\dots ,g_{u_{c+1}} \,]$ with $r$-independent
sequence of indices $(u_1,\dots,u_{c+1})$. Since $L_0=0$ and $n[g
_{u_1}, g_{u_2},\dots ,g_{u_{c+1}} \,]=0$ by property~\eqref{if-r-ind}, we obtain
that $n[[K,K]_a,\underbrace{L_b,\dots,L_b}_{m-1}
]=0$, so that we can put $m-1=w$. If $n\leq N(c,q)$, then there is
a positive integer $k<n$ such that $a+kb=0$ in
$\Bbb{Z}/n\Bbb{Z}$. Then
$[L_a,\underbrace{L_b,\dots,L_b}_{k} ]\subseteq L_0=0$, so that we can put $m-1=N(c,q)$.

Now suppose that $(n,b)\ne 1$, and let $s$ be a prime dividing
both $n$ and $b$. If $s$ also divides $a$, then the result follows
by induction on $|F|$ applied to the Lie ring $C_L(\f
^{n/s})=\sum _iL_{si}$ containing both $L_a$ and $L_b$: this Lie subring
is $FH$-invariant and $\f $ acts on it as an
automorphism of order $n/s$ without non-trivial fixed points. The
basis of this induction is the case of $n=|F|$ being a prime,
where necessarily $(n,b)=1$, which has already been dealt with above.
The functions $l=l(c,q)$ and $m=m(c,q)$ remain the same in the induction step,
so no dependence on $F$ arises.

It remains to consider the case where the prime $s$ divides
both $b$ and $n$ and does not divide $a$.
Using the same notation $x_{r^ik} =x_k^{h^i}$ (under Index Convention),  we
have
$$\Big[u_a+u_{ra}+\dots +u_{r^{q-1}a},\,\underbrace{
(v_b+v_{rb}+\dots +v_{r^{q-1}b}),\dots , (w_b+w_{rb}+\dots +w_{r^{q-1}b})}_{c}\,\Big]=0$$
for any $c$ elements $v_b,\dots ,w_b \in L_b$, because the sums are  elements of $C_L(H)$.

By \eqref{prim} any two indices $r^ia,\,r^ja$ here are different modulo $s$, while all the indices
above the brace are divisible by $s$, which divides $n$.
Hence by the `almost linear independence' \eqref{almost-ind} we also have
\begin{equation}\label{mod-s3}
n\Big[u_a,\, \underbrace{ (v_b+v_{rb}+\dots
+v_{r^{q-1}b}),\dots , (w_b+w_{rb}+\dots +w_{r^{q-1}b})}_{c}\,\Big]=0.
\end{equation}
Let $Z$ denote the additive subgroup of
$\sum _{i=0}^{q-1}L_{r^ib}$ generated by all the sums $x_b+x_{rb}+\dots
+x_{r^{q-1}b}$ over $x_b\in L_b$ (in general $Z$ may not contain all the fixed points of $H$ on $\sum _{i=0}^{q-1}L_{r^ib}$).
Then \eqref{mod-s3} means that
$$n\Big[L_a,\, \underbrace{Z,\dots ,
Z}_{c}\,\Big]=0.$$
 Applying $\f ^j$ we also obtain
\begin{equation}\label{mod-s4}
n\Big[L_a,\,
\underbrace{Z^{\f ^j},\dots , Z^{\f
^j}}_{c}\,\Big]=0.\end{equation}

We now apply Lemma~\ref{vander} with  $M=L_b+L_{rb}+\dots
+L_{r^{q-1}b}$ and $m=q$ to $w=v_b+v_{rb}+\dots
+v_{r^{q-1}b}\in Z$ for any $v_b\in L_b$.
As a result, $n^{l_0}L_b\subseteq \sum _{j=0}^{q-1}Z^{\f
^j}$ for some $(c,q)$-bounded number $l_0$.

We now claim that
$$n^{l_0(q(c-1)+1)+1}[[K,K]_a,\, \underbrace{L_b,\dots ,
L_b}_{q(c-1)+1}]=n[[K,K]_a,\, \underbrace{n^{l_0}L_b,\dots ,
n^{l_0}L_b}_{q(c-1)+1}]\subseteq [[K,K],\,[K,K]].$$
Indeed, after replacing $n^{l_0}L_b$  with $\sum
_{j=0}^{q-1}Z^{\f ^j}$ and expanding the sums, in each
commutator of the resulting linear combination we can freely
permute modulo $[[K,K],\,[K,K]]$ the entries $Z^{\f ^j}$.
Since there are sufficiently many of them, we
can place  at least $c$ of the same
$Z^{\f ^{j_0}}$ for some $j_0$ right after $[L,L]_a$ at the beginning,
which gives $0$ by \eqref{mod-s4}.
\end{proof}

We now prove part (b) of the theorem. Since $L$ is solvable of
$(c,q)$-bounded derived length by~(a), we can use induction on the
derived length of $L$. If $L$ is
abelian, there is nothing to prove. Assume that $L$ is metabelian, which is
the main part of the proof.
Consider $K=\sum L_i$. Let $l=l(c,q)$ and $m=m(c,q)$ be as in Lemma~\ref{l_b-metab2};  put
$g=(m-1)(q^{c+1}+c)+2$. For any  sequence of $g$ non-zero elements
$(a_1,\dots,a_g)$ in $\Bbb Z/n\Bbb Z$ consider the commutator
$[[K,K]_{a_1},L_{a_2},\dots,L_{a_g}]$.
If the sequence
$(a_1,\dots,a_g)$ contains an $r$-independent subsequence of
length $c+1$ that starts with $a_1$, by
permuting the $L_{a_i}$ we can assume that $a_1,\dots,a_{c+1}$ is the
$r$-independent subsequence. Then $n[[K,K]_{a_1},L_{a_2},\dots,L_{a_g}]=0$  by \eqref{if-r-ind}.
If the sequence $(a_1,\dots,a_g)$ does not contain an $r$-independent subsequence of
length $c+1$  starting with $a_1$, then by
Lemma~\ref{rigid}  the sequence $(a_1,\dots,a_g)$
contains at most $q^{c+1}+c$ different values. The number $g$ was
chosen big enough to guarantee that either the value of $a_1$
occurs in $(a_1,\dots,a_g)$ at least $m+1$ times or, else, another
value, different from $a_1$, occurs at least $m$ times. Using that
$[x,y,z]=[x,z,y]$ for $x\in [K,K]$ we can assume that $a_2=\dots=a_{m+1}$, in which
case it follows from Lemma~\ref{l_b-metab2} that
$n^l[[K,K]_{a_1},L_{a_2},\dots,L_{a_g}]=0$. Thus, in any case $n^l\g _{g+1}(K)=0$.
Clearly, we can also choose a $(c,q)$-bounded number $l_1$ such that
$\g _{g+1}(n^{l_1}K)=0$. Since $nL\subseteq K$, we obtain $\g _{g+1}(n^{l_1+1}L)=0$, as required.

Now suppose that the derived length of $L$ is at least 3. By the
induction hypothesis, $\g _{g'+1}(n^{l_2}[L,L])=0$ for some
$(c,q)$-bounded numbers $g'$ and $l_2$. Since $g\geq 4$, we can
choose a $(c,q)$-bounded number $y$ such that $y(g+1-4)\geq
(g+1)(l_1+1)$ $\Leftrightarrow$ $y(g+1)\geq (l_1+1)(g+1)+4y$.
We set $x=\max\{ y,\, l_2/2\}$. Then for $M=n^{x}L$ we have both
\begin{align*}
\g _{g+1}(M) &=\g _{g+1}(n^{x}L)=n^{x(g+1)}\g _{g+1}(L)\\
                           &\subseteq n^{(l_1+1)(g+1)+4x}\g _{g+1}(L)\subseteq n^{4x}[[L,L],[L,L]]\\
                           &=[[M,M],[M,M]]
 \end{align*}
by the metabelian case above, and
$$\g _{g'+1}([M,M])= \g _{g'+1}(n^{2x}[L,L])\subseteq \g _{g'+1}(n^{l_2}[L,L])=0.$$
Together, these give the nilpotency of $M=n^xL$ of $(c,q)$-bounded class by the Lie
ring analogue \eqref{chao} of P.~Hall's theorem.
\medskip

We now proceed with nilpotency of $L$ itself in cases (i)--(v) of part (3) .

\textit{Case} (i). The case of a Lie algebra was already settled in Theorem~\ref{liealg}.

\textit{Cases} (iii) \textit{and} (iv). In these cases we have $L=nL=n^iL$ for any $i$. By
part~(b),
$$\gamma_{v+1}(L)=\gamma_{v+1}(n^{u}L)=0.$$

\textit{Case} (v). Let $n=p^k$ for a prime $p$. Then
$$\gamma_{v+1}((p^k)^{u}L)=p^{ku(v+1)}\g _{v+1}(L)=0$$
by part~(b), so that the additive group $\gamma_{v+1}(L)$ is a $p$-group. Since in this case
$F$ is a cyclic $p$-group acting fixed-point-freely, we must have $\gamma_{v+1}(L)=0$.

\textit{Case} (ii). Since here the additive group of $L$ is periodic, it decomposes
into the direct sum of Sylow subgroups $L=\bigoplus _pT_p$, which are $FH$-invariant ideals satisfying
$[T_q, T_p]=0$ for different prime numbers $p$,~$q$. It
is sufficient to prove that every ideal $T_p$
is nilpotent of $(c,q)$-bounded class. So we can assume that
the additive group of $L$ is a $p$-group. If $p$ does not divide $n$,
then $nL=L$ and the assertion follows from~(iv). If $p$ divides
$n$, then the Hall $p'$-subgroup $\langle \chi\rangle$ of $F$ acts
fixed-point-freely on~$L$, so that $L$ admits the Frobenius group of
automorphisms $\langle\chi\rangle H$ with $C_L(\chi)=0$. Since
$|\chi|L=L$, it remains to apply~(iv) to $\langle\chi\rangle H$.
\end{proof}

\subsubsection*{Finite groups}

As a consequence we obtain our main result on the bound for the nilpotency class of a finite group
with a metacyclic Frobenius group of automorphisms.

\begin{theorem}\label{group-nilpotency}
Suppose that a finite group $G$ admits a Frobenius group of automorphisms $FH$ with
cyclic kernel $F$ of order $n$ and complement $H$ of
order $q$ such that $C_G(F)=1$ and $C_G(H)$ is nilpotent of
class $c$. Then $G$ is nilpotent
of $(c,q)$-bounded class.
\end{theorem}

The assumption of the kernel $F$ being cyclic is
essential: see examples at the end of the section.

\begin{proof}
The group $G$ is nilpotent by Theorem~\ref{t-orn}(c).
We have to bound the nilpotency class of $G$ in terms of $c$ and
$q$. Consider the associated Lie ring of the group~$G$
\begin{equation*}
L(G)=\bigoplus_{i=1}^m\gamma_i/\gamma_{i+1},
\end{equation*}
where $m$ is the
nilpotency class of $G$ and the $\gamma_i$ are terms of the
lower central series of $G$. The nilpotency class of $G$ coincides
with that of $L(G)$. The action of the group $FH$ on $G$ naturally
induces an action of $FH$ on $L(G)$. By the definition of the
associated Lie ring,
\begin{equation*}
C_{L(G)}(H)=\bigoplus_i C_{\gamma_i/\gamma_{i+1}}(H),
 \end{equation*}
which by  Theorem \ref{t-fp} is equal to
\begin{equation*}
\bigoplus_i
C_{\gamma_i}(H)\gamma_{i+1}/\gamma_{i+1}.
\end{equation*}
The Lie ring products in $L(G)$ are defined for elements of the
$\gamma_i/\gamma_{i+1}$ in terms of the images of the group
commutators and then extended by linearity. If $c$ is the
nilpotency class of $C_G(H)$, then any group commutator of weight
$c+1$ in elements of $C_G(H)$ is trivial; hence any Lie ring
commutator of weight $c+1$ in elements of
$C_{\gamma_i}(H)\gamma_{i+1}/\gamma_{i+1}$ is also trivial. Since
these elements generate $C_{L(G)}(H)$, this subring is nilpotent
of class at most $c$. Because $F$ acts fixed-point-freely on every
quotient $\gamma_i/\gamma_{i+1}$ by Lemma~\ref{l-car}, it
follows that $C_{L(G)}(F)=0$.

We now deduce from Theorem \ref{liering}(c) that
$L(G)$ is nilpotent of $(c,q)$-bounded class. Since the nilpotency class of
$G$ is equal to that of $L(G)$, the result follows.
\end{proof}

\subsubsection*{Examples}

\begin{example}
 The simple 3-dimensional Lie algebra $L$ of characteristic ${}\ne \nobreak 2$
with basis $e_1,e_2,e_3$ and structure   constants  $[e_1,e_2]=e_3$, \ $[e_2,e_3]=e_1$, \ $[e_3,e_1]=e_2$
admits the Frobenius group of automorphisms $FH$ with $F$ non-cyclic of order $4$ and $H$ of order $3$
such that $C_L(F)=0$ and $C_L(H)$ is one-dimensional. Namely, $F=\{1,f_1,f_2,f_3\}$, where $f_i(e_i)=e_i$ and
$f_i(e_j)=-e_j$ for $i\ne j$, and $H=\langle h\rangle$ with $h(e_i)=e_{i+1 \,({\rm mod}\,3)}$.
\end{example}

\begin{example}
 One can also modify  the above example to produce nilpotent Lie
rings $L$ of unbounded derived length admitting the same
non-metacyclic Frobenius  group of automorphisms $FH$ of order
$12$ such that $C_L(F)=0$ and $C_L(H)$ is abelian
(`one-dimensional'). Namely, let the additive group of $L$ be the
direct sum of three copies of $\Z /p^m\Z$ for some prime $p\ne 2$
with generators $e_1,e_2,e_3$ and let the structure constants of
$L$ be $[e_1,e_2]=pe_3$, \ $[e_2,e_3]=pe_1$, \ $[e_3,e_1]=pe_2$.
The action of $FH$ is `the same': $F=\{1,f_1,f_2,f_3\}$, where
$f_i(e_i)=e_i$ and $f_i(e_j)=-e_j$ for $i\ne j$, and $H=\langle
h\rangle$ with $h(e_i)=e_{i+1 \,({\rm mod}\,3)}$. Then $C_L(F)=0$ and $C_L(H)=\langle e_1+e_2+e_3\rangle$.
It is easy to see
that $L$ is nilpotent of class exactly $m$, while the derived length is approximately $\log m$.
\end{example}

\begin{example} If in the preceding example we choose $p>m$ , then the Lazard
correspondence based on the `truncated' Baker--Campbell--Hausdorff
formula (\cite{la}; also see, for example, \cite[\S\,10.2]{khubook2})
transforms $L$ into a finite $p$-group $P$ of the same
derived length as $L$, which admits the same group of
automorphisms $FH$ with $C_P(F)=1$ and $C_P(H)$ cyclic.
\end{example}

\bigskip

{\footnotesize
\noindent\textit{Acknowledgments.} The second  author was partly supported
by the Programme of Support of Leading Scientific Schools of the Russian Federation (grant NSh-3669.2010.1).
The third author was supported by CNPq-Brazil.

}

\end{document}